\documentclass[12pt,oneside]{amsart}
\usepackage{amssymb}
\usepackage{amsmath}
\usepackage{amsthm,amsfonts}
\usepackage{graphicx}
\usepackage{hyperref}

\def\Re{{\mathrm{Re}}\hspace{0.5mm}}

\def\Ai{{\mathrm{Ai}}}
\def\Bi{{\mathrm{Bi}}}
\def\i{\mathrm{i}}
\def\e{\mathrm{e}}
\renewcommand\O[1]{\mathcal{O}\left(#1\right)}

\newtheorem{thm}{\hskip\parindent Theorem}
\newtheorem{lem}{\hskip\parindent Lemma}[section]

\newtheorem{rem}{\hskip\parindent Remark}
\theoremstyle{definition}

\numberwithin{equation}{section}

\begin{document}
\title
{Linear Difference Equations with a Transition Point at the Origin}
\author{Lihua Cao}
\address{Department of Mathematics, Shenzhen University, Guangdong, 518060, China
and
Department of Mathematics, City University of Hong Kong, Kowloon, Hong Kong}
\email{macaolh@szu.edu.cn}

\author{Yutian Li}
\address
{Institute of Computational and Theoretical Studies, and Department of Mathematics,
Hong Kong Baptist University, Kowloon, Hong Kong}
\email{yutianli@hkbu.edu.hk}

\date{}

\maketitle

\begin{abstract}
A pair of linearly independent asymptotic solutions are constructed for the second-order linear difference equation
\begin{equation*}
P_{n+1}(x)-(A_{n}x+B_{n})P_{n}(x)+P_{n-1}(x)=0,
\end{equation*}
where $A_n$ and $B_n$ have asymptotic expansions of the form
\begin{equation*}
A_n\sim n^{-\theta}\sum_{s=0}^\infty\frac{\alpha_s}{n^s},\qquad
B_n\sim\sum_{s=0}^\infty\frac{\beta_s}{n^s},
\end{equation*}
with $\theta\neq0$ and $\alpha_0\neq0$ being real numbers, and $\beta_0=\pm2$.
Our result hold uniformly for the scaled variable $t$ in an infinite interval containing the transition point $t_1=0$,
where $t=(n+\tau_0)^{-\theta} x$ and $\tau_0$ is a small shift.
In particular,
it is shown how the Bessel functions $J_\nu$ and $Y_\nu$ get involved in the uniform asymptotic expansions of the solutions to
the above three-term recurrence relation. As an illustration of the main result,
we derive a uniform asymptotic expansion for the orthogonal polynomials associated with
the Laguerre-type weight $x^\alpha\exp(-q_mx^m)$, $x>0$, where $m$ is a positive integer, $\alpha>-1$ and $q_m>0$.
\end{abstract}



\section{Introduction}
Orthogonal polynomials play an important role in many branches of
mathematical physics, for instance, quantum mechanics, scattering
theory and statistical mechanics. A major topic in orthogonal
polynomials is the study of their asymptotic behavior as the
degree grows to infinity. Since the classical orthogonal
polynomials (Hermite, Laguarre and Jacobi) all satisfy a
second-order linear differential equation, their asymptotic
behavior can be obtained from the WKB approximation or the turning
point theory \cite{Olver}. For discrete orthogonal polynomials
(e.g., Charlier, Meixner and Krawtchouk), one can use their
generating function to obtain a Cauchy integral representation and
then apply the steepest descent method or its extensions
\cite{Olver,Wong}. However, there are orthogonal polynomials that
neither satisfy any differential equation nor have integral
representations. A powerful method, known as the nonlinear steepest descent
method for Riemann-Hilbert problems, has recently been developed
that can be applied to such polynomials. Papers that deserve
special mention include Deift \textit{et al.}~\cite{Deift}, Bleher and Its~\cite{BleherIts},
Kriecherbauer and McLaughlin~\cite{Krie}, Baik \textit{et al.}~\cite{Baik},
Ou and Wong~\cite{OW} and Zhou \textit{et al.}~\cite{Zhao}.
Despite the great success achieved by this powerful method, the
Riemann-Hilbert analysis depends heavily on the analyticity of the weight functions,
and the argument and results obtained by the method often appear in a very complicated manner.

In our view, a more natural approach to derive asymptotic
expansions for orthogonal polynomials is to develop an asymptotic
theory for linear second-order difference equations, just in the
same way as Langer, Cherry, Olver had done for linear second-order differential
equations; see the definitive book by Olver~\cite{Olver}.
Our view is based on the fact that any sequence of
orthogonal polynomials satisfies a three-term recurrence relation
of the form
\begin{equation}\label{eq:1.1}
p_{n+1}(x)=(a_{n}x+b_{n})p_{n}(x)-c_{n}p_{n-1}(x),\qquad
n=1,2,\dots ,
\end{equation}
where $a_{n}$, $b_{n}$ and $c_{n}$ are constants; see~\cite[p.43]{Szego}.
If $x$ is a fixed number,
then this recurrence relation is equivalent to a second-order linear difference equation of the form
\begin{equation}\label{eq:1.2}
y(n+2)+n^{p}a(n)y(n+1)+n^{q}b(n)y(n)=0,
\end{equation}
where $p$ and $q$ are integers.
We assume that the coefficient functions $ a(n)$ and $b(n)$ have asymptotic expansions
\begin{equation}\label{eq:1.3}
a(n)\sim\sum_{s=0}^{\infty }\frac{a _{s}}{n^{s}},\qquad
b(n)\sim\sum_{s=0}^{\infty }\frac{b _{s}}{n^{s}},
\end{equation}
where $a _{0}\neq 0$ and $b_{0}\neq 0$.
When $p=q=0$,
asymptotic solutions to this equation are classified by the roots of the \emph{characteristic equation}
\begin{equation} \label{eq:1.4}
\rho ^{2}+a_{0}\rho +b_{0}=0.
\end{equation}
If $\rho _{1}\neq \rho _{2},$ \textit{i.e.}, $a_{0}^{2}\neq 4b_{0},$ then
Birkhoff~\cite{Birkhoff1911} showed that (\ref{eq:1.2}) has two
linearly independent solutions, both of the form
\begin{equation}\label{eq:1.5}
y(n)\sim \rho ^{n}n^{\alpha }\sum_{s=0}^{\infty
}\frac{c_{s}}{n^{s}},\qquad n\rightarrow \infty ,
\end{equation}
where
\begin{equation}\label{eq:1.6}
\alpha =-\frac{a_{1}\rho+b_{1}}{2\rho ^{2}+a_{0}\rho
}=\frac{a_{1}\rho +b_{1}}{2b_{0}+a_{0}\rho },
\end{equation}
and the coefficients $c_{s}$ can be determined recursively.
This construction fails when and only when $\rho_{1}=\rho _{2}$,
\textit{i.e.}, when $a_{0}^{2}=4b_{0}$. Motivated by the
terminologies in differential equations \cite[p.230]{Olver}, we
shall call series of the form (\ref{eq:1.5}) \emph{normal series}.
If $\rho _{1}=\rho _{2}$, but $2b_{1}\neq a_{0}a_{1}$,
then Adams~\cite{Adams} constructed two linearly independent power series solutions,
which we will call \emph{subnormal series}.

Adams also studied the case
when $2b_{1}=a_{0}a_{1},$ but his analysis was incomplete, as
pointed out by Birkhoff~\cite{Birkhoff1930}. A powerful asymptotic
theory for difference equations was later given by Birkhoff~\cite{Birkhoff1930},
Birkhoff and Trjitzinsky~\cite{BirkhoffTrjitzinsky}.
However, their analysis has been considered too complicated and even impenetrable.
A more accessible approach to the results mentioned above has been
given in more recent years by Wong and Li~\cite{WongLi1992a,WongLi1992b}.
In their first paper \cite
{WongLi1992a}, they presented recursive formulas for the coefficients in the power series solutions
when $p=q=0$ in (\ref{eq:1.2}), obtained both norm and subnormal series solutions.
Their second paper \cite{WongLi1992b} dealt with
the more general case, namely, equation (\ref{eq:1.2}) with the
exponents $p$ and $q$ being not both zero.
Numerically computable error bounds for these asymptotic series can be found in our recent paper
\cite{CaoLiZhang}.

Returning to the three-term recurrence relation (\ref{eq:1.1}), we
note that the results mentioned thus for apply only when $x$ is a
fixed number. When $x$ is a variable and allowed to vary, the
roots of the characteristic equation associated with
(\ref{eq:1.1}) may coalesce, and normal series solutions may
collapse to become subnormal series. Not much work has been done
in this area until just recently.
In a series of papers
\cite{ww1,ww2,ww3}, Wang and Wong have derived asymptotic
expansions for the solutions to (\ref{eq:1.1}), which hold
``uniformly" for $x$ in infinite intervals. They first define a
sequence $\left\{ K_{n}\right\} $ recursively by
$K_{n+1}/K_{n}=c_{n},$ with $K_{0}$ and $K_{1}$ depending on the
particular sequence of polynomials,
where $c_n$ is one of the coefficients in equation (\ref{eq:1.1}). Then they put $A_{n}\equiv
a_{n}K_{n}/K_{n+1},$ $B_{n}\equiv b_{n}K_{n}/K_{n+1}$ and
$P_{n}(x)\equiv p_{n}(x)/K_{n},$ so that (\ref {eq:1.1}) becomes
\begin{equation}\label{eq:1.7}
P_{n+1}(x)-(A_{n}x+B_{n})P_{n}(x)+P_{n-1}(x)=0.
\end{equation}
The coefficients $A_n$ and $B_n$ are assumed to have
asymptotic expansions of the form
\begin{equation}\label{eq:1.8}
A_n\sim n^{-\theta}\sum_{s=0}^\infty\frac{\alpha_s}{n^s},\qquad
B_n\sim\sum_{s=0}^\infty\frac{\beta_s}{n^s},
\end{equation}
where $\theta$ is a real number and $\alpha_0\neq0$.
If $\tau_{0}$ is a constant and $N :=n+\tau _{0},$ then the expansions
in (\ref{eq:1.8}) can of course be recasted in the form
\begin{equation}\label{eq:1.9}
A_n\sim N^{-\theta}\sum_{s=0}^\infty\frac{\alpha_s'}{N^s},\qquad
B_n\sim\sum_{s=0}^\infty\frac{\beta_s'}{N^s}.
\end{equation}
In (\ref{eq:1.7}), we now set $x:=N^{\theta }t$ and
$P_{n}=\rho^{n}.$ Substituting (\ref{eq:1.9}) into
(\ref{eq:1.7}) and letting $n\rightarrow \infty $ (and hence $N
\rightarrow \infty $), we are led to the \emph{characteristic equation}
\begin{equation} \label{eq:1.10}
\rho ^{2}-(\alpha _{0}'t+\beta _{0}')\rho +1=0.
\end{equation}
(Note that $\alpha _{0}'=\alpha _{0}$ and $\beta_{0}'=\beta _{0}$).
The roots of this equation coincide when $t=t_{i}$, $i=1,2$, where
\begin{equation}\label{eq:1.11}
\alpha _{0}'t_{1}+\beta _{0}'=+2,\qquad
\alpha _{0}'t_{2}+\beta _{0}'=-2.
\end{equation}

The values of $t_{1}$ and $t_2$ play an important role in the asymptotic
theory of the three-term recurrence relation (\ref{eq:1.7}), and
they correspond to the transition points (\textit{i.e.}, turning points or singularities)
occurring in differential equations \cite{Olver}. For this reason,
Wang and Wong~\cite{ww2} also called them \emph{transition points}
of difference equations. Furthermore, they indicated that in terms of the exponent
$\theta$ in (\ref{eq:1.8}) and the transition point $t_{1}$,
there are three cases to be considered; namely, (i) $\theta \neq
0$ and $t_{1}\neq 0$, (ii) $\theta \neq 0$ and $t_{1}=0$, and
(iii) $\theta =0.$ In \cite{ww2}, they considered case (i), which
turns out to correspond to the turning-point problem for
second-order linear differential equations \cite[p.392]{Olver}.
They also showed that the asymptotic expansions of the solutions
involve the Airy functions $\Ai(\cdot)$, $\Bi(\cdot)$ and their
derivatives, and furthermore that the expansions hold uniformly
for $t$ in the infinite interval $\left[ 0,\infty \right) $. In
\cite{ww3}, they studied case (iii), and found that the
approximants of the asymptotic solutions are Bessel functions or
modified Bessel functions. 
The above results give two linearly indecent solutions in each case, however, the linear combination coefficients 
could not be determined until a recent groundbreaking progress of Wang and Wong~\cite{WangXSWong}, 
in which they solved the problem of determining the coefficients from the recurrence relation only.
We should mention that, the WKB approximations for difference equations have also been studied in recent years
by Costin and Costin~\cite{CostinCostin}, Geronimo~\cite{Geronimo}, Geronimo, Bruno and Van Assche~\cite{GOV}, Van Assche and Geronimo~\cite{VAG}, etc. 
All of these authors dealt with only case $(\mathrm{i})$, \textit{i.e.}, the Airy-type expansions or nonuniform asymptotics away from the turning points. 
It is worth noting that the determining of linear combination coefficients was treated by Van Assche and Geronimo~\cite{VAG} for the asymptotic approximations in the outer region. 

The purpose of this paper is to provide a solution to the problem
in case (ii), which is the only case that have been left out in
Wang and Wong's investigation of the asymptotic solutions to the
three-term recurrence relation (\ref{eq:1.7}). In this case, we
shall show that the approximants of the asymptotic solutions are
also Bessel functions or modified Bessel functions, but
the order of these Bessel functions and the transformation used (\textit{i.e.}, $\zeta(t)$ in (\ref{eq:4.9}) below)
differ from the one used in case (iii).
As an illustration of the main result,
we derive a uniform asymptotic expansion for the orthogonal polynomials associated with
the Laguerre-type weight $x^\alpha\exp(-q_mx^m)$, $x>0$, where $m$ is a positive integer, $\alpha>-1$ and $q_m>0$.

\section{Motivation leading to the expansion}

Turing point (or transition point) theory for second-order linear difference equations
was well established in \cite {ww2,ww3}. We simply state the procedure below.

Let $\tau_0:=-\alpha_1/(\alpha_0\theta)$ and $N:=n+\tau_0$. Then,
\begin{equation}\label{eq:2.1}
A_nx + B_n=n^{-\theta}\sum_{s=0}^{\infty}\frac{\alpha_s}{n^s}x+\sum_{s=0}^{\infty}\frac{\beta_s}{n^s}
:=\sum_{s=0}^{\infty}\frac{\alpha_s' t+\beta_s'}{N^s}.
\end{equation}
A slight computation yields
\begin{equation}\label{eq:2.2}
\alpha_0'=\alpha_0,\quad\alpha_1'=0,\quad\beta_0'=\beta_0,\quad\beta_1'=\beta_1,\quad\beta_2'=\beta_2+\beta_1\tau_0.
\end{equation}
If $\beta_0=2$ (or $-2$), it follows from (\ref{eq:1.11}) that one of the transition points is zero.
Without loss of generality, we may assume that $t_1=0<t_2$ and $\beta_0=2$, $\alpha_0<0$. (For other cases, see Remark~\ref{rem:1} below.)
Through out this paper, we assume that $\beta_1=0$, so that $\beta_1'=0$ and
\begin{equation}\label{eq:2.3}
\alpha_1't_1+\beta_1'=\beta_1=0.
\end{equation}
This assumption was used in the previous papers on transition point theory; see \cite[(2.7)]{ww2} and \cite[(2.5)]{ww3}.
(In most of the classical cases, $\beta_1=0$ and hence $\beta_1'=0$.
Also, we note that in \cite{DingleMorgan1} Dingle and Morgan have assumed a more strict condition that
$\alpha_{2s+1}=\beta_{2s+1}=0$ for $s=0,1,2\dots$.)

In what follows we will investigate the uniform
asymptotic expansions of solutions to (\ref{eq:1.7}) near the transition point $t_1=0$.
To this end, we seek a formal solution to (\ref{eq:1.7}) of the form
\begin{equation}\label{eq:2.4}
P_n(x)=\sum_{s=0}^\infty \chi_s(\xi)N^{-s},\qquad \xi=N^2\eta(t),
\end{equation}
for $t$ in a neighbourhood of $t_1=0$, where $\eta(t)$ is an increasing
function with $\eta(0)=0$. Since $x:=N^\theta t$ is fixed, when we change $n$ to $n\pm1$,
we must in the same time change $t$ to $t_\pm$, where
\begin{equation}\label{eq:2.5}
t_\pm=\left(1\pm\frac{1}{N}\right)^{-\theta}t.
\end{equation}
As a consequence, it follows from (\ref{eq:2.4}) that
\begin{equation}\label{eq:2.6}
P_{n\pm1}(x)=\sum_{s=0}^\infty \chi_s\left[(N\pm1)^2\eta\left(\Big(1\pm\frac{1}{N}\Big)^{-\theta}t\right)\right](N\pm1)^{-s}.
\end{equation}
For convenience, we also introduce the notations
\begin{equation} \label{eq:2.7}
Q_\pm(\xi):=(N\pm1)^2\eta\left[\Big(1\pm\frac{1}{N}\Big)^{-\theta}t\right]
\end{equation}
and
\begin{equation}\label{eq:2.8}
\Psi(\xi):=A_nx+B_n=\sum_{s=0}^\infty\frac{\alpha_s't+\beta_s'}{N^s}:=\sum_{s=0}^\infty\frac{T_s(\xi)}{N^s}.
\end{equation}
Substituting (\ref{eq:2.4}) and (\ref{eq:2.6}) into (\ref{eq:1.7}), we get
\begin{equation}\label{eq:2.9}
\sum_{s=0}^\infty\left\{\chi_s[Q_+(\xi)]\Big(1+\frac{1}{N}\Big)^{-s}+\chi_s[Q_-(\xi)]\Big(1-\frac{1}{N}\Big)^{-s}-\Psi(\xi)\chi_s(\xi)\right\}\frac{1}{N^s}=0.
\end{equation}
Noting that $t=\eta^{-1}(N^{-2}\xi)$, Taylor expansion gives
\begin{equation}\label{eq:2.10}
t=\sum_{j=0}^\infty\frac{\eta^{-1(j)}(0)}{j!}\left(\frac{\xi}{N^2}\right)^j=
\frac{1}{\eta'(0)}\frac{\xi}{N^2}+\O{N^{-4}}.
\end{equation}
Substituting the last equation into (\ref{eq:2.8}) yields,
\begin{equation}\label{eq:2.11}
T_0(\xi)=\beta_0'=\beta_0,\quad T_1(\xi)=\beta_1'=0,\quad T_2(\xi)=\frac{\alpha_0'\xi}{\eta'(0)}+\beta_2'.
\end{equation}
By expanding $\eta(t_\pm)$ at $t$, we obtain from (\ref{eq:2.5}) and (\ref{eq:2.7})
\begin{equation}\label{eq:2.12}
Q_{\pm}(\xi)=\xi \mp\frac{(\theta -2)\xi}{N}+\frac{(\theta ^{2}-3\theta+2)\xi }{2N^{2}}+\O{N^{-3}}:=\sum_{s=0}^\infty\frac{Q_s^{\pm}(\xi)}{N^s}.
\end{equation}
In equation (\ref{eq:2.9}), the coefficient of $N^{-s}$ vanishes for each $s\geq0$; \textit{i.e.},
\begin{equation}\label{eq:2.13}
\chi_s[Q_+(\xi)]\Big(1+\frac{1}{N}\Big)^{-s}+\chi_s[Q_-(\xi)]\Big(1-\frac{1}{N}\Big)^{-s}-\Psi(\xi)\chi_s(\xi)=0.
\end{equation}
In particular, for $s=0$, expanding the first two terms at $\xi$ and using (\ref{eq:2.8}) and (\ref{eq:2.12}) show that $\chi_{0}(\xi)$ satisfies Bessel's equation
\begin{equation}\label{eq:2.14}
\frac{d^{2}\chi_{0}}{d\xi ^{2}}+\left( \frac{\theta -1}{\theta
-2}\right) \frac{ 1}{\xi }\frac{d\chi_{0}}{d\xi }-\frac{1}{(\theta
-2)^{2}\xi ^{2}}\left( \frac{\alpha_{0}'\xi }{\eta'(0)}+\beta _{2}'\right) \chi_{0}=0,
\end{equation}
where $\theta\neq2$. For simplicity, we shall assume $0<\theta<2$. The analysis for the case $\theta<0$ or $\theta>2$ is very similar.
Thus, $\chi_{0}(\xi )$ can be expressed in terms
of Bessel functions:
\begin{equation*}
\chi_{0}(\xi )=C_{1}\xi ^{\frac{1}{2(2-\theta )}}J_{\nu }(b\xi ^{1/2})+C_{2}\xi
^{\frac{1}{2(2-\theta )}}Y_{\nu }(b\xi ^{1/2}),
\end{equation*}
where $C_1$, $C_2$ are two constants,
$b= \sqrt{\frac{-4\alpha _{0}'}{(\theta-2)^{2}\eta'(0)}}$ and
\begin{equation}\label{eq:2.15}
\nu =\sqrt{\frac{1+4\beta _{2}'}{(\theta-2)^{2}}}.
\end{equation}
Moreover, each of the subsequent coefficient functions $\chi_{s}(\xi )$, $ s=1,2,\dots$, in
(\ref{eq:2.4}) satisfies an inhomogeneous Bessel equation. This
suggests that instead of (\ref{eq:2.4}), we might try the formal
series solution
\begin{equation}\label{eq:2.16}
P_{n}(x)=Z_{\nu }(N\zeta)\sum_{s=0}^{\infty }\frac{A_{s}(\zeta )}{N^s}
+Z_{\nu +1}(N\zeta)\sum_{s=0}^{\infty }\frac{B_{s}(\zeta )}{N^s}
\end{equation}
motivated from the differential equation theory,
where we have set $\zeta(t)=[\eta(t)]^{\frac12}$ to simplify the notations.
Here $\zeta(t)>0$ for $t>0$ and $\i\zeta(t)<0$ for $t<0$.
In (\ref{eq:2.16}), $Z_{\nu}(\cdot )$ could be any solution of Bessel's equation

\begin{equation}\label{eq:2.17}
y''+\frac{1}{x}y'+\left(1-\frac{\nu^2}{x^2}\right)y=0.
\end{equation}

We state the main result of this paper in the following theorem.

\begin{thm}\label{thm:1}
Assume that the coefficients $A_n$ and $B_n$ in the recurrence relation (\ref{eq:1.7}) are real,
and have asymptotic expansions given in (\ref{eq:1.8}) with $\theta\neq0,2$ and $\beta_0=2$.
Let $t_1=0$ be a transition point defined in (\ref{eq:1.11}),
and the function $\zeta(t)$ be given as in (\ref{eq:4.9}) and (\ref{eq:4.10}).
Then, for each nonnegative integer $p$,
(\ref{eq:1.7}) has a pair of linearly independent solutions,
\begin{equation}\label{eq:2.18}
\begin{aligned} P_n(N^\theta t)=&
N^\frac12\left(\frac{4\zeta^2}{4-(\alpha_0't+\beta_0')^2}\right)^\frac14
\left[J_\nu(N\zeta)\sum_{s=0}^p\frac{
\widetilde{A}_s(\zeta)}{N^s}\right.\\
&\hspace{4.6cm}+\left.J_{\nu+1}(N\zeta)
\sum_{s=0}^p\frac{\widetilde{B}_s(\zeta)}{N^s}
+\varepsilon_n^p\right]
\end{aligned}
\end{equation}
and
\begin{equation}\label{eq:2.19}
\begin{aligned} Q_n(N^\theta t)=&
N^\frac12\left(\frac{4\zeta^2}{4-(\alpha_0't+\beta_0')^2}\right)^\frac14
\left[W_\nu(N\zeta)\sum_{s=0}^p\frac{\widetilde{A}_s(
\zeta)}{N^s}\right.\\
&\hspace{4.6cm}+\left.W_{\nu+1}(N\zeta)
\sum_{s=0}^p\frac{\widetilde{B}_s(\zeta)}{N^s}
+\delta_n^p\right],
\end{aligned}
\end{equation}
where $W_\nu(x):=Y_\nu(x)-\i J_\nu(x)$, $N=n+\tau_0$, $\tau_0=-\alpha_1/(\alpha_0\theta)$ and $\nu$ is given in (\ref{eq:2.15}). The error terms satisfy
\begin{equation}\label{eq:2.20}
|\varepsilon_n^p|\leq
\frac{M_p}{N^{p+1}}
\big[|J_\nu(N\zeta)|+|J_{\nu+1}(N\zeta)|\big]
\end{equation}
and
\begin{equation}\label{eq:2.21}
|\delta_n^p|\leq
\frac{M_p}{N^{p+1}}
\big[|W_\nu(N\zeta)|+|W_{\nu+1}(N\zeta)|\big]
\end{equation}
for $-\infty<t\leq t_2-\sigma$, $M_p$ being a positive constant and $\sigma$ being an arbitrary  positive constant.
The coefficients $\widetilde{A}_0(\zeta)=1$ and $\widetilde{B}_0(\zeta)=0$
and other $\widetilde{A}_s(\zeta)$ and $\widetilde{B}_s(\zeta)$
can be determined successively from their predecessors;
see (\ref{eq:4.26}) and (\ref{eq:4.27}).
\end{thm}

\begin{rem}\label{rem:1}
For the case
$\alpha_0>0$ and $\beta_0=-2$, the two transition points $t_1$ and $t_2$ again satisfy $t_1=0<t_2$.
Put $\mathcal{P}_n(x):=(-1)^nP_n(x)$. Theorem~\ref{thm:1} then applies to $\mathcal{P}_n(x)$.
If the two transition points satisfy $t_2<0=t_1$, we set $x:=-N^\theta t$, instead of $x:=N^\theta t$.
Theorem~\ref{thm:1} again holds with $P_n(N^\theta t)$ replaced by $P_n(-N^\theta t)$
(in the case $\alpha_0>0$, $\beta_0=2$) or $(-1)^n P_n(-N^\theta t)$ (in the case $\alpha_0<0$, $\beta_0=-2$).
\end{rem}

\section{A preliminary lemma}

As we have noted in above sections,  $x:=N^\theta t$ is fixed in  equation (\ref{eq:1.7}).
Hence, as $N$ being replaced by $N\pm1$, the two functions
$Z_\nu(N \zeta)$ and $Z_{\nu+1}(N \zeta)$ in (\ref{eq:2.16}) automatically
change to $Z_\nu\left[(N\pm1) \zeta(t_\pm)\right]$ and $Z_{\nu+1} \left[(N\pm1) \zeta(t_\pm)\right]$,
where $t_\pm$ are given in (\ref{eq:2.5}).
 The following lemma plays a crucial role in the derivation of the formal series solution in (\ref{eq:2.16}).

\begin{lem}\label{lem:3.1}
Let $Z_{\nu}(x)$ be any solution of Bessel's equation
\begin{equation} \label{eq:3.1}
y''+\frac{1}{x}y'+\left(1-\frac{\nu^{2}}{x^{2}}\right)
=0.
\end{equation}
Then, we have
\begin{equation} \label{eq:3.2}
\begin{aligned}
Z_{\nu }\left\{( N\pm 1)\zeta(t_\pm) \right\}
=Z_\nu(N\zeta)  G_\pm\Big(\zeta,\frac{1}{N}\Big)
+  Z_{\nu+1}(N\zeta)  H_\pm\Big(\zeta,\frac{1}{N}\Big)
\end{aligned}
\end{equation}
and
\begin{equation} \label{eq:3.3}
\begin{aligned}
Z_{\nu+1}\left\{(N\pm1)\zeta(t_\pm)\right\}
=  Z_\nu(N\zeta)  L_\pm\Big(\zeta,\frac{1}{N}\Big)
+Z_{\nu+1}(N\zeta)  K_\pm\Big(\zeta,\frac{1}{N}\Big),
\end{aligned}
\end{equation}
where
\begin{equation} \label{eq:3.4}
G_\pm\Big(\zeta,\frac{1}{N}\Big)\sim
\sum_{s=0}^{\infty}(\pm1)^s\frac{  G_{s}(\zeta)}{N ^{s}},\qquad
H_\pm\Big(\zeta,\frac{1}{N}\Big)\sim
\sum_{s=0}^{ \infty}(\pm1)^{s-1}\frac{  H_{s}(\zeta)}{N ^{s}}
\end{equation}
and
\begin{equation}\label{eq:3.5}
L_\pm\Big(\zeta,\frac{1}{N}\Big)\sim
\sum_{s=0}^{\infty}(\pm1)^{s-1}\frac{L_{s}(\zeta)}{N ^{s}},\qquad
K_\pm\Big(\zeta,\frac{1}{N}\Big)\sim
\sum_{s=0}^{ \infty}(\pm1)^s\frac{  K_{s}(\zeta)}{N ^{s}},
\end{equation}
the expansions being uniformly valid with respect to bounded $t$.
\end{lem}

\begin{proof} Let
\begin{equation}\label{eq:3.6}
w(N,\zeta,u) =\left( 1+\frac{u}{N }\right) ^{\frac12}Z_{\nu}\left[(N+u)\zeta\right].
\end{equation}
A straightforward calculation gives
\begin{equation}\label{eq:3.7}
\frac{\partial w}{\partial
u}=\frac{\frac{1}{2}+\nu}{N}\left(1+\frac{u}{N}\right)^{-\frac{1}{2}}Z_\nu
-\left(1+\frac{u}{N}\right)^{\frac{1}{2}}\zeta Z_{\nu+1}
\end{equation}
and
\begin{equation}\label{eq:3.8}
\frac{\partial^2 w}{\partial u^2}
=\left(\frac{\nu^2-\frac{1}{4}}{(N+u)^2}-\zeta^2\right)w,
\end{equation}
where we have made use of the difference-differential relations of
Bessel functions
$$
Z_\nu'(x)=\frac{\nu}{x}Z_\nu(x)-Z_{\nu+1}(x)
$$
and
$$
Z_{\nu+1}'(x)=Z_\nu(x)-\frac{\nu+1}{x}Z_{\nu+1}(x);
$$
see \cite[\S 10.6]{NIST}. The Taylor expansion of $w(N,\zeta,u)$ at $u=0$ is
\begin{equation}\label{eq:3.9}
w(N,\zeta,u) =w(N,\zeta,0) +u\frac{\partial w}{\partial u}(N,\zeta,0) +\frac{u^{2}}{2!}\frac{\partial ^{2}w}{\partial u^{2}}(N,\zeta,0) +\cdots.
\end{equation}
In view of (\ref{eq:3.6})-(\ref{eq:3.8}), the above series can be rearranged as
\begin{equation}\label{eq:3.10}
\left( 1+\frac{u}{N }\right)
^{\frac{1}{2}}Z_{\nu}\left[(N+u)\zeta\right]=Z_\nu(N\zeta)\widetilde{G}(N;\zeta,u)
+Z_{\nu+1}(N\zeta)\widetilde{H}(N;\zeta,u),
\end{equation}
where
\begin{equation}\label{eq:3.11}
\widetilde{G}(N;\zeta,0)
=1,\qquad\widetilde{H}(N;\zeta,0)=0,
\end{equation}
\begin{equation}\label{eq:3.12}
\frac{
\partial \widetilde{G}}{\partial u}(N;\zeta,0) =\frac{\nu+\frac{1}{2}}{N}\quad \text{and}\quad\frac{
\partial \widetilde{H}}{\partial u}(N;\zeta,0) =-\zeta.
\end{equation}
Differentiating (\ref{eq:3.10}) with respect to $u$ twice yields
\begin{equation}\label{eq:3.13}
\frac{\partial ^{2}\widetilde{G}}{\partial u^{2}}
=\left(\frac{\nu^2-\frac{1}{4}}{(N+u)^2}-\zeta^2\right)\widetilde{G},\quad
\widetilde{G}|_{u=0}=1,\quad
\left.\frac{\partial\widetilde{G}}{\partial
u}\right|_{u=0}=\frac{\frac{1}{2}+\nu}{N}
\end{equation}
and
\begin{equation}\label{eq:3.14}
\frac{\partial ^{2}\widetilde{H}}{\partial u^{2}}
=\left(\frac{\nu^2-\frac{1}{4}}{(N+u)^2}-\zeta^2\right)\widetilde{H},\quad
\widetilde{H}|_{u=0}=0,\quad\left.\frac{\partial\widetilde{H}}{\partial
u}\right|_{u=0}=-\zeta.
\end{equation}
We try formal solutions to (\ref{eq:3.13}) and (\ref{eq:3.14}) of the following forms
\begin{equation}\label{eq:3.15}
\widetilde{G}(N;\zeta,u) \sim \sum_{s=0}^{\infty
}\frac{\widetilde{ G}_{s}(\zeta,u) }{N ^{s}}\quad\text{ and
}\quad\widetilde{H}(N; \zeta,u) \sim \sum_{s=0}^{\infty
}\frac{\widetilde{H}_{s}(\zeta,u)}{N ^{s}},
\end{equation}
where the coefficients $\widetilde{G}_s$ can be determined recursively by the equations
\begin{equation}\label{eq:3.16}
\frac{\partial^2 \widetilde G_0}{\partial u^2}=-\zeta^2\widetilde G_0,
\quad \widetilde G_0(0,\zeta)=1,\quad \frac{\partial \widetilde G_0}{\partial u}(0,\zeta)=0,
\end{equation}
\begin{equation}\label{eq:3.17}
\frac{\partial^2 \widetilde G_1}{\partial u^2}=-\zeta^2 \widetilde G_1,
\quad \widetilde G_1(0,\zeta)=0,\quad \frac{\partial \widetilde G_1}{\partial u}(0,\zeta)=\frac{1}{2}+\nu,
\end{equation}
\begin{equation}\label{eq:3.18}
\frac{\partial^2 \widetilde G_s }{\partial u^2}=-\zeta^2 \widetilde G_s+
\left(\nu^2-\frac{1}{4}\right)\sum_{j=2}^{s}(j-1)(-u)^{j-2}\widetilde G_{s-j}, \qquad s \ge 2,
\end{equation}
and
\begin{equation}\label{eq:3.19}
\widetilde G_s(0,\zeta)=0,\quad \frac{\partial \widetilde G_s}{\partial u}(0,\zeta)=0,\qquad s\geq 2.
\end{equation}
The solutions of (\ref{eq:3.16})--(\ref{eq:3.19}) are
\begin{equation}\label{eq:3.20}
\widetilde{G}_0=\cos (\zeta u),\qquad
\widetilde{G}_1=\frac{\frac{1}{2}+\nu}{\zeta}\sin(\zeta u)
\end{equation}
and
\begin{equation}\label{eq:3.21}
\widetilde{G}_s=
\frac{\nu^2-\frac{1}{4}}{\zeta}\int_{0}^{u}\Big(
\sum_{j=2}^{s}(-1)^j(j-1)\phi ^{j-2} \widetilde{G}_{s-j}\Big)
\sin \big((u-\phi)\zeta\big) d\phi
\end{equation}
for $s\geq2$. Similarly, we also have
\begin{equation}\label{eq:3.22}
\widetilde{H}_0=-\sin(\zeta u),
\qquad\widetilde{H}_1 =0
\end{equation}
and
\begin{equation}\label{eq:3.23}
\widetilde{H}_s=\frac{\nu^2-\frac{1}{4}}{\zeta}\int_{0}^{u}\Big(
\sum_{j=2}^{s}(-1)^j(j-1)\phi ^{j-2} \widetilde{H}_{s-j}\Big)
\sin \big( (u-\phi)\zeta\big) d\phi
\end{equation}
for $s\ge2$.
By induction, we can prove that when $\i\zeta<0$,
$$
|\widetilde G_s(u,\zeta)|\leq(|\nu|+1)^s|u|^s\widetilde G_0(u, \zeta),\quad
|\widetilde H_s(u,\zeta)|\leq(|\nu|+1)^s|u|^s|\widetilde H_0(u,\zeta)|,
$$
and that when $\zeta >0$,
$$ |\widetilde
G_s(u,\zeta)|\leq(|\nu|+1)^s|u|^s,\quad |\widetilde
H_s(u,\zeta)|\leq(|\nu|+1)^s|u|^{s}.
$$
Thus, the formal series in (\ref{eq:3.15}) are uniformly convergent for
any bounded $u$ and sufficiently large $N$.

To show that (\ref{eq:3.2}) follows from (\ref{eq:3.10}), we simply set
$$
(N+u)\zeta=(N\pm1)\zeta(t_\pm)
$$
\textit{i.e.}, we choose
\begin{equation}\label{eq:3.24}
u=u_\pm:=(N\pm1) \frac{\zeta(t_\pm)}{\zeta(t)}-N.
\end{equation}
Furthermore, we let
\begin{equation}\label{eq:3.25}
G_\pm\Big(\zeta,\frac1N\Big)=\left(1+\frac{u_\pm}{N}\right)^{-\frac12}\widetilde{G}(N;\zeta,u_\pm)
\end{equation}
and
\begin{equation}\label{eq:3.26}
H_\pm\Big(\zeta,\frac1N\Big)=\left(1+\frac{u_\pm}{N}\right)^{-\frac12}\widetilde{H}(N;\zeta,u_\pm).
\end{equation}

Recall that $t_\pm$ is given in (\ref{eq:2.5}). Expand $t_\pm$ into power series in $\frac1N$,
and then expand $\zeta(t_\pm)$ into Taylor series at $t$ and rearrange them also into power series in $\frac1N$.
From (\ref{eq:3.24}), it follows that
\begin{equation}\label{eq:3.27}
u_\pm=\sum_{s=0}^\infty (\pm1)^{s+1}\frac{u_s}{N^s},
\end{equation}
where
\begin{equation}\label{eq:3.28}
u_0=1-\theta t\frac{\zeta'(t)}{\zeta(t)},\qquad
u_1=\frac{\theta(\theta-1)t}{2}\frac{\zeta'(t)}{\zeta(t)}
+\frac{\theta^2t^2}{2}\frac{\zeta''(t)}{\zeta(t)}.
\end{equation}
Substituting (\ref{eq:3.27}) into (\ref{eq:3.15}), we obtain
\begin{equation}\label{eq:3.29}
  G_\pm\Big(\zeta,\frac{1}{N}\Big)
=\left(1+\frac{u_\pm}{N}\right)^{-\frac12}\sum_{s=0}^{\infty}\frac{\widetilde{G}_s(\zeta,u_\pm )}{N^{s}}:=\sum_{s=0}^{\infty}\frac{  G_{s}^\pm(\zeta)}{N^{s}},
\end{equation}
\begin{equation}\label{eq:3.30}
H_\pm\Big(\zeta,\frac{1}{N}\Big)
=\left(1+\frac{u_\pm}{N}\right)^{-\frac12}\sum_{s=0}^{\infty}\frac{\widetilde{H}_s(\zeta,u_\pm )}{N^{s}}:=\sum_{s=0}^{\infty}\frac{  H_{s}^\pm(\zeta)}{N^{s}}.
\end{equation}
Simple calculation shows that  the first two terms in the expansions of (\ref{eq:3.29}) and (\ref{eq:3.30}) are, respectively, given by
\begin{equation}\label{eq:3.31}
  G_0^\pm(\zeta) =\cos(\zeta u_0),\quad
G_1^\pm(\zeta) =
\pm\frac{\frac{1}{2}+\nu-\zeta^2u_1}{\zeta}\sin(\zeta u_0)\mp\frac12u_0\cos(\zeta u_0)
\end{equation}
and
\begin{equation}\label{eq:3.32}
H_0^\pm(\zeta) =\mp \sin(\zeta u_0), \qquad
H_1^\pm(\zeta) =-u_1\zeta \cos(\zeta u_0)+\frac{u_0}{2}\sin(\zeta u_0).
\end{equation}

Equation (\ref{eq:3.3}) and the expansions in (\ref{eq:3.5}) can be established in similar manners. Moveover, we have
\begin{equation}\label{eq:3.33}
K_0^\pm(\zeta)=\cos (\zeta u_0), \quad
K_1^\pm(\zeta)=\mp\frac{\frac{1}{2}+\nu+\zeta^2
u_1}{\zeta}\sin(\zeta u_0)\mp\frac{u_0}{2}\cos (\zeta u_0)
\end{equation}
and
\begin{equation}\label{eq:3.34}
L_0^\pm(\zeta)=\pm \sin(\zeta u_0), \qquad
L_1^\pm(\zeta)=u_1\zeta\cos(\zeta u_0)-\frac{u_0}{2}\sin(\zeta u_0).
\end{equation}

By  mathematical induction, one can show that
\begin{equation}\label{eq:3.35}
G_s^-(\zeta)=(-1)^{s}  G_s^+(\zeta),\qquad
H_s^-(\zeta)=(-1)^{s-1}H_s^+(\zeta)
\end{equation}
and
\begin{equation}\label{eq:3.36}
K_s^-(\zeta)=(-1)^{s}  K_s^+(\zeta),\qquad
L_s^-(\zeta)=(-1)^{s-1}L_s^+(\zeta).
\end{equation}
When we consider the case `+', for simplicity we shall drop the subscripts and superscripts in all above equations.
Thus, (\ref{eq:3.35}) and (\ref{eq:3.36}) give
\begin{equation}\label{eq:3.37}
G_\pm\Big(\zeta,\frac{1}{N}\Big)=
G\Big(\zeta,\pm\frac{1}{N}\Big),\qquad
H_\pm\Big(\zeta,\frac{1}{N}\Big)=\pm
H\Big(\zeta,\pm\frac{1}{N}\Big)
\end{equation}
and
\begin{equation}\label{eq:3.38}
  K_\pm\Big(\zeta,\frac{1}{N}\Big)=
K\Big(\zeta,\pm\frac{1}{N}\Big),\qquad
L_\pm\Big(\zeta,\frac{1}{N}\Big)=\pm
L\Big(\zeta,\pm\frac{1}{N}\Big).
\end{equation}
This completes the proof of Lemma~\ref{lem:3.1}.
\end{proof}

\section{Formal asymptotic solutions}

Let $\zeta^2(t)=\eta(t)$ be an increasing function with $\zeta(0)=0$;
see the sentence following (\ref{eq:2.16}).
We try a formal series solution to (\ref{eq:1.7}) of the form
\begin{equation}\label{eq:4.1}
 P_n(N^\theta t)=Z_\nu(N\zeta)\sum_{s=0}^\infty\frac{A_s(\zeta)}{N^s}
 +Z_{\nu+1}(N\zeta)\sum_{s=0}^\infty\frac{B_s(\zeta)}{N^s},
\end{equation}
where $Z_\nu$ is a solution of Bessel's equation. For
convenience, we put
\begin{equation}\label{eq:4.2}
A\Big(\zeta,\frac{1}{N }\Big) :=\sum_{s=0}^{\infty }\frac{A_{s}(\zeta) }{N ^{s}},\qquad
B\Big(\zeta,\frac{1}{N }\Big) :=\sum_{s=0}^{\infty}\frac{B_{s}(\zeta)}{N ^s}
\end{equation}
and
\begin{equation}\label{eq:4.3}
\Psi \Big(t,\frac{1}{N }\Big):=A_nx+B_n
=\sum_{s=0}^{\infty }\frac{\alpha'_{s}t+\beta'_{s}}{N^s};
\end{equation}
cf. (\ref{eq:2.8}). Using (\ref{eq:2.16}) and Lemma~\ref{lem:3.1}, we have
\begin{equation} \label{eq:4.4}
\begin{aligned} P_{n\pm1}(x)=&Z_\nu(N\zeta)
\left\{G\Big(\zeta,\pm\frac{1}{N}\Big)A\Big(\zeta(t_\pm),\frac{1}{N\pm1}\Big)\right.
\\
&\qquad\pm\left.  L\Big(\zeta,\pm\frac{1}{N}\Big)B\Big(\zeta(t_\pm),\frac{1}{N\pm1}\Big)\right\}
\\
&+Z_{\nu+1}(N\zeta)
\left\{K\Big(\zeta,\pm\frac{1}{N}\Big)B\Big(\zeta(t_\pm),\frac{1}{N\pm1}\Big)\right.
\\
&\qquad \pm\left. H\Big(\zeta,\pm\frac{1}{N}\Big)A\Big(\zeta(t_\pm),\frac{1}{N\pm1}\Big)\right\};
\end{aligned}
\end{equation}
see also (\ref{eq:3.37}) and (\ref{eq:3.38}).
Substituting (\ref{eq:4.1})--(\ref{eq:4.4}) into (\ref{eq:1.7}) and
matching the coefficients of $Z_\nu$ and $Z_{\nu+1}$, we get
\begin{equation}\label{eq:4.5}
\begin{aligned}
&G\Big(\zeta,\frac{1}{N}\Big)A\Big[\zeta(t_+),\frac{1}{N+1}\Big]
+ L\Big(\zeta,\frac{1}{N}\Big)B\Big[\zeta(t_+),\frac{1}{N+1}\Big]\\
+&G\Big(\zeta,\frac{-1}{N}\Big)A\Big[\zeta(t_-),\frac{1}{N-1}\Big]
- L\Big(\zeta,\frac{-1}{N}\Big)B\Big[\zeta(t_-),\frac{1}{N-1}\Big]\\
-&\Psi\Big(t,\frac{1}{N}\Big)A\Big(\zeta,\frac{1}{N}\Big)=0
\end{aligned}
\end{equation}
and
\begin{equation}\label{eq:4.6}
\begin{aligned}
& H\Big(\zeta,\frac{1}{N}\Big)A\Big[\zeta(t_+),\frac{1}{N+1}\Big]
+K\Big(\zeta,\frac{1}{N}\Big)B\Big[\zeta(t_+),\frac{1}{N+1}\Big]\\
-& H\Big(\zeta,\frac{-1}{N}\Big)A\Big[\zeta(t_-),\frac{1}{N-1}\Big]
+K\Big(\zeta,\frac{-1}{N}\Big)B\Big[\zeta(t_-),\frac{1}{N-1}\Big]\\
-&\Psi\Big(t,\frac{1}{N}\Big)B\Big(\zeta,\frac{1}{N}\Big)=0.
\end{aligned}
\end{equation}
Letting $N\to\infty $ in (\ref{eq:4.5}) and (\ref{eq:4.6}), we obtain
\begin{equation}\label{eq:4.7}
G_0(\zeta)=K_0(\zeta)=\frac{\alpha'_0t+\beta'_0}{2},
\end{equation}
which combined with (\ref{eq:3.28}) and (\ref{eq:3.31}) yields
\begin{equation}\label{eq:4.8}
\zeta(t)-\theta t \zeta'(t)
=\pm\arccos\left(\frac{\alpha'_0t+\beta'_0}{2}\right),
\end{equation}
where the $\arccos$  function is analytically continued to $\mathbb C\setminus\{(0,-\i\infty)\cup(t_2,\infty)\}$ so that for $t<0$,
\begin{equation*}
\arccos \left(\frac{\alpha _{0}' t+\beta _{0}'}{2}\right)
=\frac{1}{\i}\log \left[ \left( \frac{\alpha _{0}'t+\beta_{0}'}{2}\right) +\i\sqrt{1-\left( \frac{\alpha_{0}'t+\beta_{0}'}{2}\right) ^{2}}\right],
\end{equation*}
where $t_2$ is given in (\ref{eq:1.11}).
Recall  that $\alpha_0'<0$ and $\beta_0'=2$; see the lines between (\ref{eq:2.2}) and (\ref{eq:2.3}).
It is straightforward to verify by direct differentiation that the function
\begin{equation}\label{eq:4.9}
\pm\zeta(t)
=\arccos \left(\frac{\alpha_{0}' t+\beta_{0}'}{2}\right)
+\alpha _{0}' t^{\frac{1}{\theta}}\int_{a}^{t}\frac{ \phi ^{-\frac{1}{\theta}}}{\sqrt{4-(\alpha _{0}' \phi +\beta_{0}' )^{2}}}d\phi
\end{equation}
is a solution to (\ref{eq:4.8}) for $t<t_{2}-\sigma$, where `$+$' sign is taken for $\theta<2$ and `$-$' sign is taken for $\theta>2$.
The choice of the lower limit of integration is just for the purpose of convergence of the integral.
For instance, we can chose $a=0$ if $\theta<0$ or $\theta>2$ and
\begin{equation}\label{eq:4.10}
 a=\left\{\begin{aligned}&-\infty,\quad t<0\\&t_2,\quad 0\leq t<t_2,
 \end{aligned}\right.\qquad\text{if $0<\theta<2$}.
\end{equation}
With this choice, it can also be shown that $\zeta(0)=0$.

Equating coefficients of like powers of $\frac1N$ in (\ref{eq:4.5}) and
(\ref{eq:4.6}) to zero, we obtain
\begin{equation}\label{eq:4.11}
\begin{aligned} &\sum_{s\leq p,\hspace{0.5mm} s\hspace{1mm}\text{even}}\hspace{1mm}
\sum_{i+m\leq s}G_i(\zeta)
\binom{-p+s}{s-i-m}\sum_{l=0}^{m}\frac{D^{l}A_{p-s}(\zeta)}{l!}t^{l}\gamma _{l,m}\\
+&\sum_{s\leq p,\hspace{0.5mm} s\hspace{1mm}\text{odd}}\hspace{1mm}
\sum_{i+m\leq s}L_i(\zeta)
\binom{-p+s}{s-i-m}\sum_{l=0}^{m}\frac{D^{l}B_{p-s}(\zeta)}{l!}t^{l}\gamma_{l,m}\\
-&\frac{1}{2}\sum_{s\leq p}\left(\alpha'_st+\beta'_s\right)\ A_{p-s}(\zeta)
=0
\end{aligned}
\end{equation}
and
\begin{equation}  \label{eq:4.12}
\begin{aligned}
&\sum_{s\leq p,\hspace{0.5mms\hspace{1mm}\text{even}}}\hspace{1mm} \sum_{i+m\leq s}K_i(\zeta)
\binom{-p+s}{s-i-m}\sum_{l=0}^{m}\frac{D^{l}B_{p-s}(\zeta)}{l!}t^{l}\gamma _{l,m}\\
+&\sum_{s\leq p,\hspace{0.5mm}s\hspace{1mm}\text{odd}} \hspace{1mm} \sum_{i+m\leq s}H_i(\zeta)
\binom{-p+s}{s-i-m}
\sum_{l=0}^{m}\frac{D^{l}A_{p-s}(\zeta)}{l!}t^{l}\gamma _{l,m}\\
-&\frac{1}{2}\sum_{s\leq p}\left(\alpha'_st+\beta'_s\right)B_{p-s}(\zeta) =0,
\end{aligned}
\end{equation}
where $D^{j}$\ denotes the $j$-th derivative with respect to $t$, \textit{i.e.},
\begin{equation} \label{eq:4.13}
D^{j}A_{l}(\zeta )=\frac{d^{j}}{dt^{j}}A_{l}(\zeta (t)),\qquad
l=0,1,2,\dots,
\end{equation}
and $\gamma _{l,m}$ is defined by
\begin{equation} \label{eq:4.14}
\sum_{m=0}^{\infty }\frac{\gamma _{l,m}}{N^{m}}:=\left[\left(
1+\frac{1}{N} \right) ^{-\theta }-1\right] ^{l}.
\end{equation}
For convenience, we define
\begin{equation}\label{eq:4.15}
\begin{aligned}
f_{p-1}(t):= &\sum_{2\leq s\leq p}\frac{\alpha_s't+\beta_s'}{2}A_{p-s}(\zeta)  \\
&-\sum_{{2\leq s\leq p} \atop {s\hspace{0.5mm}\text{even}}}
\left[ \sum_{i+m\leq s}\binom{-p+s}{s-i-m}G_{i}(\zeta)\sum_{l=0}^{m}\frac{D^{l}A_{p-s}(\zeta)}{l!}\gamma_{l,m}t^{l}\right]\\
&-\sum_{2\leq s\leq p \atop s\hspace{1mm}\text{odd}}
\left[ \sum_{i+m\leq s}\binom{-p+s}{s-i-m}L_{i}(\zeta)\sum_{l=0}^{m}\frac{D^{l}B_{p-s}(\zeta)}{l!}\gamma _{l,m}t^{l}\right]
\end{aligned}
\end{equation}
and
\begin{equation}\label{eq:4.16}
\begin{aligned}
g_{p-1}(t):= &\sum_{2\leq s\leq p}\frac{\alpha_s't+\beta_s'}{2}B_{p-s}(\zeta)  \\
&-\sum_{2\leq s\leq p\atop s\hspace{1mm}\text{even}}
\left[ \sum_{i+m\leq s}\binom{-p+s}{s-i-m}K_{i}(\zeta)\sum_{l=0}^{m}\frac{D^{l}B_{p-s}(\zeta)}{l!}\gamma_{l,m}t^{l}\right] \\
&-\sum_{2\leq s\leq p\atop s\hspace{1mm}\text{odd}}
\left[ \sum_{i+m\leq s}\binom{-p+s}{s-i-m}H_{i}(\zeta)\sum_{l=0}^{m}\frac{D^{l}A_{p-s}(\zeta)}{l!}\gamma_{l,m}t^{l}\right]
\end{aligned}
\end{equation}
for $p\geq 1$. Clearly, $f_{0}(t)=g_{0}(t)=0$. Note that
$\gamma _{0,0}=1$, $\gamma_{0,m}=0$, $ m=1,2,\dots $, and $\gamma _{1,0}=0,\gamma
_{1,1}=-\theta $. By (\ref{eq:3.31})--(\ref{eq:3.34}) and (\ref{eq:4.7}),  equations (\ref{eq:4.11})
and (\ref{eq:4.12}) can be rewritten as
\begin{equation}\label{eq:4.17}
[(1-p)L_{0}(\zeta)+L_{1}(\zeta)]B_{p-1}(\zeta )-\theta
tL_{0}(\zeta)\frac{d}{dt}B_{p-1}(\zeta)=f_{p-1}(t)
\end{equation}
and
\begin{equation}\label{eq:4.18}
[(1-p)H_{0}(\zeta)+H_{1}(\zeta)]A_{p-1}(\zeta)-\theta
tH_{0}(\zeta)\frac{d}{dt}A_{p-1}(\zeta )=g_{p-1}(t),
\end{equation}
where we have made use of the assumption that $\alpha_1=\beta_1'=0$; cf. (\ref{eq:2.2}) and (\ref{eq:2.3}).
It follows from (\ref{eq:3.31}), (\ref{eq:3.32}) and (\ref{eq:4.7}) that
\begin{equation}\label{eq:4.19}
H_{0}(\zeta )=-\sqrt{\frac{4-(\alpha _{0}'t+\beta_0')^{2}}{4}},
\qquad
\frac{dH_{0}(\zeta)}{dt}
=-\frac{\alpha _0'}{2}\frac{G_{0}(\zeta)}{H_{0}(\zeta)}.
\end{equation}
Differentiating (\ref{eq:4.8}), we get
\begin{equation}\label{eq:4.20}
(1-\theta )\zeta'(t)-\theta t\zeta''(t)=\frac{\alpha _{0}'}{2 H_{0}(\zeta)}.
\end{equation}
Combining (\ref{eq:3.28}), (\ref{eq:3.32}), (\ref{eq:3.34}), (\ref{eq:4.19}) and (\ref{eq:4.20}), we obtain
\begin{equation}\label{eq:4.21}
H_{1}(\zeta)=-\frac{\theta t}{2}\frac{d}{dt}H_{0}(\zeta )-\frac12\left[1-\theta t\frac{\zeta'}{\zeta}\right]H_{0}(\zeta )
\end{equation}
and
\begin{equation} \label{eq:4.22}
L_{1}(\zeta)=-\frac{\theta t}{2}\frac{d}{dt}L_{0}(\zeta )-\frac12\left[1-\theta t\frac{\zeta'}{\zeta}\right]L_{0}(\zeta ).
\end{equation}
Substituting (\ref{eq:4.21}) and (\ref{eq:4.22}) into (\ref{eq:4.17}) and (\ref{eq:4.18}) and
noting that $ L_{0}(\zeta )=-H_{0}(\zeta )$, we have
\begin{equation} \label{eq:4.23}
\left\{ \begin{aligned}
&\frac{d}{dt}[t^{\frac{p-1}{\theta}}\Lambda(t)B_{p-1}]
=\frac{t^{\frac{p-1}{\theta}-1}\Lambda(t)}{\theta H_0(\zeta)}f_{p-1}(t),\\
&\frac{d}{dt}[t^{\frac{p-1}{\theta}}\Lambda(t)A_{p-1}]
=-\frac{t^{\frac{p-1}{\theta}-1}\Lambda(t)}{\theta H_0(\zeta)}g_{p-1}(t),\\
\end{aligned}\right.
\end{equation}
where
\begin{equation}\label{eq:4.24}
\Lambda (t):=t^{\frac1{2\theta}}\left[ \frac{-H_{0}(\zeta)}{\zeta}\right] ^{\frac{ 1}{2}}.
\end{equation}
For $p=1$, since $f_{0}=g_{0}=0$, $\Lambda A_{0}(\zeta)$ and $\Lambda B_{0}(\zeta)$ are constants, we set
\begin{equation}\label{eq:4.25}
\Lambda A_{0}(\zeta)=1,\qquad \Lambda B_{0}=0.
\end{equation}
For $p>1$ and $0<\theta<2$
\begin{equation}\label{eq:4.26}
t^{\frac{p-1}{\theta}}\Lambda(t)B_{p-1}(\zeta)=\int_{0}^{t}\frac{s^{\frac{p-1}{\theta}-1}\Lambda(s)}{\theta  H_0(\zeta(s))}f_{p-1}(s)ds,\quad t<t_{2}-\sigma
\end{equation}
and
\begin{equation}\label{eq:4.27}
t^{\frac{p-1}{\theta}}\Lambda (t)A_{p-1}(\zeta)=-\int_{0}^{t}\frac{s^{\frac{p-1}{\theta}-1}\Lambda(s)}{\theta H_0(\zeta(s))}g_{p-1}(s)ds,\quad t<t_{2}-\sigma.
\end{equation}
Therefore, for each $p>1$, $A_{p}(t)$ and $B_{p}(t)$ can be determined
successively from their predecessors $A_{0}(t)$, $B_{0}(t)$,
$\dots $, $A_{p-1}(t)$ and $ B_{p-1}(t)$. If, in addition, we
have the assumption $\alpha _{2s+1}'=\beta_{2s+1}'=0$ for $s=0,1,\dots $, we can choose the
solutions such that $A_{2s+1}=B_{2s}=0$.

\section{Bounds for coefficients}

\begin{lem}\label{lem:5.1}
Let $\zeta(t)$ be given as in (\ref{eq:4.9}) and (\ref{eq:4.10}), and
let $ G(\zeta,\frac{1}{N})$, $H(\zeta,\frac{1}{N})$,
$L(\zeta,\frac{1}{N})$ and $ K(\zeta,\frac{1}{N})$ be given as in
Lemma~\ref{lem:3.1}. Then there exists a constant $C_s$
independent of $t$ such that
\begin{equation}\label{eq:5.1}
|G_s(\zeta)|\le C_s(|t|+1),\qquad |H_s(\zeta)|\le C_s(|t|+1),
\end{equation}
\begin{equation}\label{eq:5.2}
|K_s(\zeta)|\le C_s(|t|+1),\qquad |L_s(\zeta)|\le C_s(|t|+1)
\end{equation}
for all $t\le t_2-\sigma$, $\sigma>0$.
\end{lem}

\begin{proof} Note that the function $\zeta^2(t)$ in (\ref{eq:4.9}) is $C^\infty$ in
$(-\infty,t_2-\sigma)$, and that each $u_k$,
$k=0,1,\dots$, given in (\ref{eq:3.27}), is continuous at $t=0$. Hence
the expansions in (\ref{eq:3.4}) and (\ref{eq:3.5}) hold uniformly with
respect to $t$ in any compact subinterval of $(-\infty,t_2)$, and the
coefficients $G_s(\zeta)$ and $H_s(\zeta)$ in these expansions are
bounded for all finite $t$ away from $t_2$. Thus, to prove the two
estimates in (\ref{eq:5.1}), it suffices to show that the following two
limits exist:
\begin{equation}\label{eq:5.3}
\lim_{t\to-\infty}\frac{G_s(\zeta)}{t},\qquad\lim_{t\to-\infty}\frac{H_s(\zeta)}{t}.
\end{equation}
Replacing $Z_\nu$ by $J_\nu$ and $Y_\nu$ and choosing positive sign in (\ref{eq:3.2}), we obtain
\begin{equation}\label{eq:5.4}
\begin{aligned}
&J_{\nu }\left\{(N+1)\zeta(t_+)\right\}
=J_\nu(N\zeta)G\Big(\zeta,\frac{1}{N}\Big)+ J_{\nu+1}(N\zeta)H\Big(\zeta,\frac{1}{N}\Big),
\end{aligned}
\end{equation}
and
\begin{equation}\label{eq:5.5}
\begin{aligned}
&Y_{\nu}\left\{(N+1)\zeta(t_+)\right\}
=Y_\nu(N\zeta)G\Big(\zeta,\frac{1}{N}\Big)+ Y_{\nu+1}(N\zeta)H\Big(\zeta,\frac{1}{N}\Big).
\end{aligned}
\end{equation}
Upon solving the last two equations for $G(\zeta,\frac{1}{N})$ and
$H(\zeta,\frac{1}{N})$, we have
\begin{equation}\label{eq:5.6}
\begin{aligned}
G\Big(\zeta,\frac{1}{N}\Big)=&\frac{\pi N\zeta}{2}
\big\{J_{\nu+1}(N\zeta)Y_{\nu }\left[ (N+1)\zeta(t_+)\right]
-Y_{\nu+1}(N\zeta)J_{\nu }\left[ (N+1)\zeta(t_+) \right]\big\}
\end{aligned}
\end{equation}
and
\begin{equation}\label{eq:5.7}
\begin{aligned}
H\Big(\zeta,\frac{1}{N}\Big)=&\frac{\pi N\zeta}{2}
\big\{Y_{\nu}(N\zeta)J_{\nu }\left[(N+1)\zeta(t_+)\right]
-J_{\nu}(N\zeta)Y_{\nu }\left[ (N+1)\zeta(t_+) \right]\big\},
\end{aligned}
\end{equation}
where we have made use of the identity
\begin{equation*}
J_{\nu+1}(x)Y_\nu(x)-J_{\nu}(x)Y_{\nu+1}(x)=\frac{2}{\pi x};
\end{equation*}
see \cite[p. 244]{Olver} and \cite[eq. (10.5.2)]{NIST}. Recall the relations
between the Bessel functions and modified Bessel functions \cite[pp. 60 \& 251]{Olver}
\begin{equation}\label{eq:5.8}
J_\nu(\i x)=\e^{\nu\pi \i/2}I_\nu(x),\qquad
Y_\nu(\i x)=\e^{(\nu+1)\pi\i/2}I_\nu(x)-\frac{2}{\pi}\e^{-\nu\pi \i/2}K_\nu(x),
\end{equation}
and let
\begin{equation}\label{eq:5.9}
\xi=-\i N\zeta(t),\qquad\xi'=-\i(N+1)\zeta\left(t_+\right).
\end{equation}
Then, (\ref{eq:5.6}) can be written as
\begin{equation}\label{eq:5.10}
G\Big(\zeta,\frac{1}{N}\Big)=\xi\big[I_{\nu}(\xi')K_{\nu+1}(\xi)+K_{\nu}(\xi')I_{\nu+1}(\xi)\big].
\end{equation}
Recall the asymptotic expansions of the modified Bessel functions for large
argument
\begin{equation}\label{eq:5.11}
I_\nu(x)\sim\frac{\e^x}{(2\pi x)^{1/2}}\sum_{s=0}^\infty(-1)^s\frac{\mu_s}{x^s}
\end{equation}
and
\begin{equation}\label{eq:5.12}
K_\nu(x)\sim\left(\frac{\pi}{2x}\right)^{1/2}\e^{-x}\sum_{s=0}^\infty\frac{\mu_s}{x^s},
\end{equation}
where $\mu_0=1$; see \cite[pp. 250-251,
238]{Olver}. A combination of (\ref{eq:5.10})-(\ref{eq:5.12}) gives
\begin{equation}\label{eq:5.13}
\begin{aligned}
G\Big(\zeta,\frac{1}{N}\Big)\sim
\frac{1}{2}\left(\frac{\xi}{\xi'}\right)^\frac12&\left[\e^{\xi'-\xi}\sum_{s=0}^\infty(-1)^s\frac{\mu_s}{\xi^s}\left(\frac{\xi}{\xi'}\right)^s\sum_{l=0}^\infty\frac{\widetilde{\mu}_l}{\xi^l}\right.
\\
&\hspace{5mm}\left.+\e^{\xi-\xi'}\sum_{s=0}^\infty(-1)^s\frac{\widetilde{\mu}_s}{\xi^s}\sum_{l=0}^\infty\frac{\mu_l}{\xi^l}\left(\frac{\xi}{\xi'}\right)^l\right],
\end{aligned}
\end{equation}
where $\widetilde{\mu}_0=1$ and
$$
\widetilde{\mu}_{s}(\nu)=\mu_s(\nu+1),\qquad s\ge1.
$$
It follows from (\ref{eq:4.9}) and (\ref{eq:4.10}) that as
$t\to-\infty$
\begin{equation}\label{eq:5.14}
\zeta^{\pm1}(t)=\O{{|t|^{\pm\frac{1}{\theta}}}},\qquad\theta>{2},
\end{equation}
and
\begin{equation}\label{eq:5.15}
\zeta^{\pm1}(t)=\O{{\log^{\pm1}|t|}},\qquad 0<\theta<2
\mbox{ or } \theta<0.
\end{equation}
Using (\ref{eq:4.9}), (\ref{eq:4.10}) and (\ref{eq:5.14}), a straightforward calculation gives
\begin{equation}\label{eq:5.16}
\begin{aligned}
\xi'-\xi=&\log\frac{\alpha_0't+\beta_0'+\sqrt{(\alpha_0't+\beta_0')^2-4}}{2}\\
&+\frac{\theta\alpha_0't}{N}\int_0^1\frac{(s-1)(1+\frac{s}{N})^{-\theta-1}}{\sqrt{(\alpha_0't(1+\frac{s}{N})^{-\theta}+\beta_0')^2-4}}ds:=\sum_{j=0}^\infty\frac{\delta_j(t)}{N^j},
\end{aligned}
\end{equation}
where
$\delta_0(t)=\log[(\alpha_0't+\beta_0'+\sqrt{(\alpha_0't+\beta_0')^2-4})/2]$
and
\begin{equation}\label{eq:5.17}
\delta_j(t)=\frac{(-1)^{j}\theta}{j(j+1)}+\mathcal{O}_j\left({|t|^{-1}}\right)\qquad
\text{as $t\to-\infty$}.
\end{equation}
Note: in (\ref{eq:5.17}) we have used a subscript ``$j$'' to indicate that each $\mathcal{O}$-term depends on $j$, $j=1,2,\dots$.
Hence,
\begin{equation}\label{eq:5.18}
\frac{\xi'}{\xi}=1+\i\sum_{j=1}^\infty\frac{\delta_{j-1}(t)\zeta^{-1}(t)}{N^j}.
\end{equation}
For each $j\ge1$, the coefficient of $N^{-j}$ is
$\O{\log^{-1}|t|}$ for large negative $t$.
We define
$$
\exp\left[\sum_{j=1}^{\infty}\frac{1}{N^j}\left(\delta_j(t)-\frac{(-1)^j\theta}{j(j+1)}\right)\right]
:=\sum_{j=0}^{\infty}\frac{a_j(t)}{N^j},$$
$$
\exp\left[\sum_{j=1}^{\infty}\frac{1}{N^j}\left(-\delta_j(t)-\frac{(-1)^j\theta}{j(j+1)}\right)\right]:=\sum_{j=0}^{\infty}\frac{b_j(t)}{N^j}
$$
and
$$
\left(\frac{\xi}{\xi'}\right)^l:=\sum_{j=0}^{\infty}\frac{c_{l,j}(t)}{N^j}.
$$
It is readily verified that $a_0(t)=b_0(t)=c_{l,0}(t)=1$,
$a_j(t)=\mathcal{O}_j(|t|^{-1})$ for each $j \geq 1$ and each $l \geq 0$, and other $b_j(t)$, $c_{l,j}(t)$ are also
bounded for $t\leq t_2-\sigma$. With these notations we have from (\ref{eq:5.13}) that
$$
\begin{aligned}
2G\Big(\zeta,\frac{1}{N}\Big)\sim&
\exp\left[\delta_0(t)+\sum_{j=1}^\infty\frac{(-1)^j\theta}{j(j+1)} \cdot \frac{1}{N^j}\right]
\cdot\sum_{j=0}^{\infty}\frac{a_j(t)}{N^j}\\
&\times\sum_{s=0}^\infty \frac{\widetilde{\mu}_s}{N^s}\left(\i \zeta^{-1}(t)\right)^s
\cdot\sum_{l=0}^\infty\frac{(-1)^l\mu_l}{N^l}\left(\i\zeta ^{-1}(t)\right)^l
\sum_{m=0}^\infty\frac{c_{l+\frac12,m}(t)}{N^m}\\
&+\exp\left[-\delta_0(t)+\sum_{j=1}^{\infty}\frac{(-1)^j\theta}{j(j+1)}
\frac{1}{N^j}\right]\cdot \sum_{j=0}^{\infty}\frac{b_j(t)}{N^j}\\
&\times\sum_{s=0}^\infty(-1)^s\frac{\widetilde{\mu}_s}{N^s}\left(\i\zeta^{-1}(t)\right)^s
\cdot \sum_{l=0}^\infty \frac{\mu_l}{N^l}\left(\i\zeta^{-1}(t)\right)^l
\sum_{m=0}^\infty\frac{c_{l+\frac{1}{2},m}(t)}{N^m},
\end{aligned}
$$
which in turn gives
$$
\begin{aligned}
2G\Big(\zeta,\frac{1}{N}\Big)
\sim &\e^{\theta} \left(1+\frac{1}{N}\right)^{-(N+1)\theta}
\left[\frac{\alpha_0't+\beta_0'+\sqrt{(\alpha_0't+\beta_0')^2-4}}{2}\cdot
\sum_{p=0}^\infty \frac{w_{1,p}(t)}{N^p}\right]\\
&+\e^{\theta}
\left(1+\frac{1}{N}\right)^{-(N+1)\theta}\left[
\frac{\alpha_0't+\beta_0'-\sqrt{(\alpha_0't+\beta_0')^2-4}}{2}\cdot
\sum_{p=0}^\infty \frac{w_{2,p}(t)}{N^p}\right],
\end{aligned}
$$
where $w_{1,0}(t)=w_{2,0}(t)=1$ and for $p \ge 1$
$$
w_{1,p}(t)=\sum_{j+s+l+m=p}(-1)^{l}
a_j(t)\widetilde{\mu}_s\mu_l[\i\zeta^{-1}(t)]^{s+l}
c_{l+\frac12,m}(t),
$$
$$
w_{2,p}(t)=\sum_{j+s+l+m=p}(-1)^{s}
b_j(t)\widetilde{\mu}_s\mu_l[\i\zeta^{-1}(t)]^{s+l}
c_{l+\frac12,m}(t).
$$
Note that for each $p\geq1$, we
have $w_{1,p}=\O{1}$ and $w_{2,p}=\O{1}$. If we set
$$
\tau_p(t) =
\frac{\alpha_0't+\beta_0'+\sqrt{(\alpha_0't+\beta_0')^2-4}}{2(\alpha_0't+\beta_0')}
w_{1,p}(t)+\frac{\alpha_0't+\beta_0'-\sqrt{(\alpha_0't+\beta_0')^2-4}}{2(\alpha_0't+\beta_0')}
w_{2,p}(t),
$$
then we further obtain
\begin{equation}\label{eq:5.19}
G\Big(\zeta,\frac{1}{N}\Big) \sim
\frac{\alpha_0't+\beta_0'}{2}\e^{\theta}\left(1+\frac{1}{N}\right)^{-(N+1)\theta}\sum_{p=0}^\infty
\frac{\tau_p(t)}{N^p},
\end{equation}
where $\tau_0(t)=1$ and $\tau_p(t)=\O{1}$ for each $p \geq 1$.
Let
\begin{equation} \label{eq:5.20}
\e^{\theta}\left(1+\frac{1}{N}\right)^{-(N+1)\theta}=\exp
\left[\theta-\theta (N+1) \log
\left(1+\frac{1}{N}\right)\right]:=\sum_{s=0}^\infty
\frac{\omega_s}{N^s}.
\end{equation}
A combination of  (\ref{eq:5.19}) and (\ref{eq:5.20}) yields
\begin{equation}\label{eq:5.21}
G\Big(\zeta, \frac{1}{N}\Big) \sim
\frac{\alpha_0't+\beta_0'}{2} \sum_{s=0}^\infty \frac{1}{N^s}
\sum_{j=0}^s{\omega_{s-j}\tau_j(t)},
\end{equation}
and a comparison of (\ref{eq:3.4}) and (\ref{eq:5.21}) gives
\begin{equation}\label{eq:5.22}
G_s(\zeta) =
\frac{\alpha_0't+\beta_0'}{2}\left[\omega_s + \O {1}\right].
\end{equation}
The existence of the first limit in (\ref{eq:5.3}) is thus proved.
The existence of the second limit in (\ref{eq:5.3}) can be shown in a similar manner,
and we have thus established the two estimates in (\ref{eq:5.1}).
The proofs of the two inequalities in (\ref{eq:5.2}) are very similar to the ones given
above, and hence will not be repeated here.
This completes the proof of Lemma~\ref{lem:5.1}.
\end{proof}

Using the estimates proved in Lemma~\ref{lem:5.1}, we can provide some bounds for the
coefficient functions $A_s(\zeta)$ and $B_s(\zeta)$ in (\ref{eq:4.1}). We state
the result as follows.

\begin{lem} \label{lem:5.2}
For $s$, $j=0,1,2,\dots$ and $t<0$, there exists a constant $M_{s,j}$ such
that
\begin{equation} \label{eq:5.23}
\big|t^{j+\frac1{2\theta}}(-H_0/\zeta)^\frac{1}{2}D^jA_s(\zeta)\big|\leq
M_{s,j}(1+|t|^{-\frac{s}{\theta}})
\end{equation}
and
\begin{equation}\label{eq:5.24}
\big|t^{j+\frac1{2\theta}}(-H_0/\zeta)^\frac{1}{2}D^jB_s(\zeta)\big|\leq
M_{s,j}(1+|t|^{-\frac{s}{\theta}}),
\end{equation}
where $D^j$ denotes the $j$-th derivative with respect to $t$.
\end{lem}

\begin{proof}
Our approach is based on mathematical induction.
Since $\zeta(t)$, $G_s(\zeta)$, $H_s(\zeta)$, $L_s(\zeta)$ and $K_s(\zeta)$ are all analytic functions in the left half plane $\Re t<0$,
and since the estimates in Lemma~\ref{lem:5.1} also hold for complex $t$ in this half plane,
an application of Cauchy's integral formula shows that for $s,j=0,1,2,\dots$ and $t<0$,
there exist constants $\widetilde{N}_j$ and $N_{s,j}$ such that
\begin{equation}\label{eq:5.25}
\left|t^jD^j\zeta\right|\leq \widetilde{N}_j(1+|\zeta|),
\end{equation}
\begin{equation}
\left|t^jD^jG_s(\zeta)\right|\leq N_{s,j}(1+|t|),
\end{equation}
\begin{equation}
\left|t^jD^jH_s(\zeta)\right|\leq N_{s,j}(1+|t|),
\end{equation}
\begin{equation}
\left|t^jD^jK_s(\zeta)\right|\leq N_{s,j}(1+|t|),
\end{equation}
\begin{equation}\label{eq:5.29}
\left|t^jD^jL_s(\zeta)\right|\leq N_{s,j}(1+|t|).
\end{equation}
When $s=0$, the estimates in (\ref{eq:5.23}) and (\ref{eq:5.24}) follow from (\ref{eq:4.25}) for $j=0$.
This, together with (\ref{eq:4.17}) and (\ref{eq:4.18}), implies that (\ref{eq:5.23}) and (\ref{eq:5.24}) hold for $j=1$.
Furthermore, by differentiating (\ref{eq:4.17}) and (\ref{eq:4.18}) $j$ times with respect to $t$,
one can show that (\ref{eq:5.23}) and (\ref{eq:5.24}) also hold for $j\geq2$.
Now assume these estimates hold for $s=0,1,\dots,p-2$, $p\geq 2$.
We shall show that they are valid for $s=p-1$.
Using (\ref{eq:5.25})--(\ref{eq:5.29}), it can be shown from (\ref{eq:4.15}) and (\ref{eq:4.16}) that
\begin{equation}\label{eq:5.30}
|t^\frac{1}{2\theta}(-H_0/\zeta)^\frac12 f_{p-1}(t)|\leq C_{p,0} (1+|t|)(1+|t|^{-\frac{p-2}{\theta}})
\end{equation}
and
\begin{equation}\label{eq:5.31}
|t^\frac{1}{2\theta}(-H_0/\zeta)^\frac12 g_{p-1}(t)|\leq C_{p,0} (1+|t|)(1+|t|^{-\frac{p-2}{\theta}}),
\end{equation}
where $C_{p,0}$ is a positive constant. Combining the last two inequalities with (\ref{eq:4.27}) leads to
\begin{equation}\label{eq:5.32}
\begin{aligned}
|t^\frac{1}{2\theta}(-H_0/\zeta)^\frac12 A_{p-1}(\zeta)|
=& |t|^{-\frac{p-1}{\theta}}\left|C+\int_{-1}^t\frac{s^{\frac{p-1}{\theta}-1}}{\theta H_0(\zeta(s))}\Lambda(s)f_{p-1}(s)ds\right|\\
\leq&|t|^{-\frac{p-1}{\theta}}\left( C+
\widetilde{C}_{p,0}\left|\int_{-1}^t |s|^{\frac{p-1}{\theta}-1}(1+|s|^{-\frac{p-2}{\theta}})
ds\right|\right)\\
\leq&|t|^{-\frac{p-1}{\theta}}\left( C+
\widetilde{C}_{p,0}\int_1^{-t} (\tau^{\frac{p-1}{\theta}-1}+\tau^{\frac{1}{\theta}-1})d\tau\right)\\
\leq&M_{p-1,0}(1+|t|^{-\frac{p-1}{\theta}}).
\end{aligned}
\end{equation}
Here we have made use of (\ref{eq:4.19}). In a similar manner, we have
\begin{equation}\label{eq:5.33}
\big|t^{\frac1{2\theta}}(-H_0/\zeta)^{\frac12}B_{p-1}(\zeta)\big|\leq
M_{p-1,0}(1+|t|^{-\frac{p-1}{\theta}}).
\end{equation}
Hence, (\ref{eq:5.23}) and (\ref{eq:5.24}) hold for $s=p-1$ and $j=0$.
Using (\ref{eq:4.17}) and (\ref{eq:4.18}), one can also show that (\ref{eq:5.23}) and (\ref{eq:5.24}) are valid for $s=p-1$ and $j=1,2,\dots$.
This completes the proof of Lemma~\ref{lem:5.2} by induction.
\end{proof}

\section{Asymptotic nature of the expansion}

Since $x$ is fixed in the recurrence relation (\ref{eq:1.7}),
we let $W_\nu(x):=Y_\nu(x)-\i J_\nu(x)$ and choose
\begin{equation*}
Z_{\nu }(N\zeta )=x^\frac{1}{2\theta}J_{\nu}(N\zeta)\quad \text{and}\quad
Z_{\nu }(N\zeta )=x^\frac{1}{2\theta}W_{\nu }(N\zeta )
\end{equation*}
in (\ref{eq:4.1}) to yield two linearly independent solutions
\begin{equation}\label{eq:6.1}
\begin{aligned}
P_n(N^\theta t)=&
\left(\frac{4\zeta^2}{4-(\alpha_0' t+\beta_0')^2}\right)^\frac14
N^\frac12\left[J_\nu(N\zeta )\sum_{s=0}^p\frac{\widetilde{A}_s(\zeta)}{N^{s}}\right.\\
&\hspace{4.5cm}\left.+J_{\nu+1}(N\zeta )\sum_{s=0}^p\frac{\widetilde{B}_s(\zeta)}{N^{s}}+\varepsilon_n^p(N,t)\right]
\end{aligned}
\end{equation}
and
\begin{equation}\label{eq:6.2}
\begin{aligned}
Q_n(N^\theta t)=&
\left(\frac{4\zeta^2}{4-(\alpha_0' t+\beta_0')^2}\right)^\frac14
N^\frac12\left[W_\nu(N\zeta )\sum_{s=0}^p\frac{\widetilde{A}_s(\zeta)}{N^{s}}\right.\\
&\hspace{4.0cm}\left.+W_{\nu+1}(N\zeta )\sum_{s=0}^p\frac{\widetilde{B}_s(\zeta)}{N^{s}}+\delta_n^p(N,t)\right],
\end{aligned}
\end{equation}
where $\widetilde{A}_{s}(\zeta )=t^{\frac{1}{2\theta}}[-H_{0}(\zeta )/\zeta]^{\frac12} A_{s}(\zeta )$,
$\widetilde{B}_{s}(\zeta)=t^{\frac{1}{2\theta}}[-H_{0}(\zeta )/\zeta]^{\frac12} B_{s}(\zeta )$, and
\begin{equation}\label{eq:6.3}
\widetilde{A}_{0}(\zeta )=1,\qquad
\widetilde{B} _{0}(\zeta )=0.
\end{equation}
Here we have made use of (\ref{eq:4.19}), (\ref{eq:4.24}) and (\ref{eq:4.25}).
From (\ref{eq:4.9}) and (\ref{eq:4.10}), one can get the behavior of $\zeta(t)$ as $t\to-\infty$; cf. (\ref{eq:5.14}) and (\ref{eq:5.15}).
With this, it follows from Cauchy's integral formula that there is a constant $C_k$ such that
\begin{equation*}
\left|t^kD^k\Lambda(t)\right|\leq C_k|\Lambda(t)|,\qquad t<0.
\end{equation*}
Applying Leibniz's rule to the products $\widetilde{A}_s(\zeta)=\Lambda(t)A_s(\zeta)$ and
$\widetilde{B}_s(\zeta)=\Lambda(t)B_s(\zeta)$,
the estimates in (\ref{eq:5.23}) and (\ref{eq:5.24}) give
\begin{equation}\label{eq:6.4}
|t^{j}D^{j}\widetilde{A}_{s}(\zeta)|
\leq \widetilde{M} _{s,j}(1+|t|^{-\frac{s}{\theta }})
\end{equation}
and
\begin{equation}\label{eq:6.5}
|t^{j}D^{j}\widetilde{B}_{s}(\zeta)
|\leq \widetilde{M} _{s,j}(1+|t|^{-\frac{s}{\theta }}),
\end{equation}
where $\widetilde{M}_{s,j}$ is some other constant.
It readily follows from (\ref{eq:6.3}) that $\widetilde{A}_0(\zeta)$ and $\widetilde{B}_0(\zeta)$
are bounded for $t<t_2-\sigma$. Assume that $\widetilde{A}_s(\zeta)$ and $\widetilde{B}_s(\zeta)$
are bounded for $t<t_2-\sigma$ and $s\leq p$.
Using an argument similar to that given for (\ref{eq:5.30})-(\ref{eq:5.33}),
we have from (\ref{eq:4.15}) and (\ref{eq:4.16}) that
\begin{equation*}
|t^\frac{1}{2\theta}(-H_0/\zeta)^\frac12 f_{p+1}(t)|\leq C_{p+1,0} (1+|t|),
\end{equation*}
\begin{equation*}
|t^\frac{1}{2\theta}(-H_0/\zeta)^\frac12 g_{p+1}(t)|\leq C_{p+1,0}  (1+|t|).
\end{equation*}
As a consequence, it can be shown that (\ref{eq:5.23}) and (\ref{eq:5.24}) can be improved to read
\begin{equation*}
\big|t^{j+\frac1{2\theta}}(-H_0/\zeta)^\frac{1}{2}D^jA_s(\zeta)\big|\leq M_{s,j}
\end{equation*}
and
\begin{equation*}
\big|t^{j+\frac1{2\theta}}(-H_0/\zeta)^\frac{1}{2}D^jB_s(\zeta)\big|\leq M_{s,j}
\end{equation*}
for $t<0$ and $s,j\geq 0$. In particular, these two estimates hold for $j=0$ and $s=p+1$;
that is, $\widetilde{A}_{p+1}(\zeta)$ and $\widetilde{B}_{p+1}(\zeta)$ are bounded for $t<0$.
By the continuity of $\zeta(t)$, $G_s(\zeta)$ and $H_s(\zeta)$ we conclude that
$\widetilde{A}_{p+1}(\zeta)$ and $\widetilde{B}_{p+1}(\zeta)$ are also bounded
for $t<t_2-\sigma$, thus completing the induction argument.

Now we shall show that (\ref{eq:1.7}) has two linearly independent
solutions $ P_{n}(x)$ and $Q_{n}(x)$ with (\ref{eq:6.1}) and
(\ref{eq:6.2}) being their uniform asymptotic expansions for $t<t_{2}-\sigma$.
We need to show that (\ref{eq:2.20}) and (\ref{eq:2.21}) hold for $t<t_2-\sigma$.

For convenience, we introduce the notations
\begin{equation}\label{eq:6.6}
A_{p}\Big( \zeta ,\frac{1}{N}\Big) :=
\left(\frac{4\zeta^2}{4-(\alpha_0' t+\beta_0')^2}\right)^\frac14 N^\frac12
\sum_{s=0}^{p}\frac{\widetilde{A}_{s}(\zeta )}{N^{s}},
\end{equation}
\begin{equation}\label{eq:6.7}
B_{p}\Big( \zeta ,\frac{1}{N}\Big) :=
\left(\frac{4\zeta^2}{4-(\alpha_0' t+\beta_0')^2}\right)^\frac14 N^\frac12
\sum_{s=0}^{p}\frac{\widetilde{B}_{s}(\zeta)}{N^{s}},
\end{equation}
\begin{equation}\label{eq:6.8}
r_n^p(x):=J_\nu(N\zeta )A_p\Big(\zeta,\frac{1}{N}\Big)
+J_{\nu+1}(N\zeta )B_p\Big(\zeta,\frac{1}{N}\Big)
\end{equation}
and
\begin{equation}\label{eq:6.9}
s_n^p(x):=W_\nu(N\zeta )A_p\Big(\zeta,\frac{1}{N}\Big)
+W_{\nu+1}(N\zeta)B_p\Big(\zeta,\frac{1}{N}\Big).
\end{equation}

By Lemma~\ref{lem:3.1} we have
\begin{equation}\label{eq:6.10}
r_{n+1}^{p}(x)-(A_{n}x+B_{n})r_{n}^{p}(x)+r_{n-1}^{p}(x)
=\frac{R_{n}^{p}(x)}{N^{p+\frac{3}{2}}},
\end{equation}
where $x=N^{\theta }t$ and the inhomogeneous term is given by
\begin{equation} \label{eq:6.11}
\frac{R_n^p(x)}{N^{p+\frac{3}{2}}}=J_\nu(N\zeta )F_{1,n}(x)+J_{\nu+1}(N\zeta )F_{2,n}(x)
\end{equation}
with
\begin{equation} \label{eq:6.12}
\begin{aligned}
F_{1,n}(x)=
&G\Big(\zeta,\frac{1}{N}\Big)A_p\Big(\zeta(t_+),\frac{1}{N+1}\Big)
+G\Big(\zeta,-\frac{1}{N}\Big)A_p\Big(\zeta(t_-),\frac{1}{N-1}\Big)\\
&+L\Big(\zeta,\frac{1}{N}\Big)B_p\Big(\zeta(t_+),\frac{1}{N+1}\Big)
-L\Big(\zeta,-\frac{1}{N}\Big)B_p\Big(\zeta(t_-),\frac{1}{N-1}\Big) \\
&-\Psi\Big(t,\frac{1}{N}\Big)A_p\Big(\zeta,\frac{1}{N}\Big)\\
\end{aligned}
\end{equation}
and
\begin{equation} \label{eq:6.13}
\begin{aligned}
F_{2,n}(x)=&
K\Big(\zeta,\frac{1}{N}\Big)B_p\Big(\zeta(t_+),\frac{1}{N+1}\Big)
+K\Big(\zeta,-\frac{1}{N}\Big)B_p\Big(\zeta(t_-),\frac{1}{N-1}\Big)\\
&+ H\Big(\zeta,\frac{1}{N}\Big)A_p\Big(\zeta(t_+),\frac{1}{N+1}\Big)
-H\Big(\zeta,-\frac{1}{N}\Big)A_p\Big(\zeta(t_-),\frac{1}{N-1}\Big)\\
&-\Psi\Big(t,\frac{1}{N}\Big)B_p\Big(\zeta,\frac{1}{N}\Big);
\end{aligned}
\end{equation}
cf. (\ref{eq:4.5}) and (\ref{eq:4.6}).
Note that the series
\begin{equation*}
A\Big( \zeta ,\frac{1}{N}\Big) :=\sum_{s=0}^{\infty}\frac{A_{s}(\zeta )}{ N^{s}}
\qquad \text{and}\qquad
B\Big(\zeta,\frac{1}{N}\Big) :=\sum_{s=0}^{\infty }\frac{B_{s}(\zeta)}{N^{s}}
\end{equation*}
in (\ref{eq:4.2}) are formal solutions of (\ref{eq:4.5}) and
(\ref{eq:4.6}). Since $ A_{p}\left( \zeta ,\frac{1}{N}\right) $ and
$B_{p}\left( \zeta ,\frac{1}{N} \right) $ can be rewritten as
\begin{equation} \label{eq:6.14}
A_{p}\Big( \zeta ,\frac{1}{N}\Big) :=x^{\frac{1}{2\theta }}\sum_{s=0}^{\infty }\frac{A_{s}^{\ast }(\zeta )}{N^{s}}
\quad\text{and}\quad
B_{p}\Big( \zeta ,\frac{1}{N}\Big):=x^{\frac{1}{2\theta } }\sum_{s=0}^{\infty }\frac{B_{s}^{\ast}(\zeta )}{N^{s}}
\end{equation}
with $A_{s}^{\ast }(\zeta )=A_{s}(\zeta )$, $B_{s}^{\ast }(\zeta
)=B_{s}(\zeta )$, for $s\leq p$ and $A_{s}^{\ast }(\zeta )=0$,
$B_{s}^{\ast }(\zeta )=0$ for $s>p$, terms with powers of $1/N$
less than or equal to $p+1 $ in the expansions $F_{1,n}(x)$ and $F_{2,n}(x)$ all vanish,
\textit{i.e.}, if we write
$(N^{\theta }t)^{-\frac{1}{2\theta}}F_{1,n}(x)\sim \sum_{s=0}^{\infty}f_{1,s}(t)N^{-s}$ and
$(N^{\theta }t)^{-\frac{1}{2\theta}}F_{2,n}(x)\sim \sum_{s=0}^{\infty}f_{2,s}(t)N^{-s}$,
then $f_{1,s}(t) $=$f_{2,s}(t)=0$ for $0\leq s\leq p+1$.
(For $s=0,1,\dots,p$, all we need are recurrence relations obtained from (\ref{eq:4.5}) and (\ref{eq:4.6});
for $s=p+1$, there are extra terms and we have to show that these terms indeed all cancel out.)
Using (\ref{eq:5.6})--(\ref {eq:5.12}), it can be shown that there is a constant $C$ such that
\begin{equation*}
\Big|G\Big(\zeta ,\frac{1}{N}\Big)\Big|\leq C(1+|t|),\qquad
\Big|H\Big(\zeta ,\frac{1 }{N}\Big)\Big|\leq C(1+|t|)
\end{equation*}
and
\begin{equation*}
\Big|K\Big(\zeta ,\frac{1}{N}\Big)\Big|\leq C(1+|t|),\qquad
\Big|L\Big(\zeta ,\frac{1 }{N}\Big)\Big|\leq C(1+|t|)
\end{equation*}
for all $t<t_{2}-\sigma$. Recall that
$\widetilde{A}_s(\zeta)$ and $\widetilde{B}_s(\zeta)$ are bounded in $t<t_2-\sigma$ for $s\leq p$.
Hence by Lemma~\ref{lem:5.1}, it can be proved that there
exists a constant $C_{p}$ such that
\begin{equation}\label{eq:6.15}
|t^\frac{1}{2\theta}(-H_{0}/\zeta)^{\frac{1}{2}}(N^\theta t)^{-\frac{1}{2\theta}}F_{1,n}(N^{\theta }t)|\leq C_{p}(1+|t|)/N^{p+2}
\end{equation}
and
\begin{equation} \label{eq:6.16}
|t^\frac{1}{2\theta}(-H_{0}/\zeta)^{\frac{1}{2}}(N^\theta t)^{-\frac{1}{2\theta}}F_{2,n}(N^{\theta }t)|\leq C_{p}(1+|t|)/N^{p+2}
\end{equation}
for all $t<t_{2}-\sigma$. Then, the last two inequalities, together with (\ref{eq:6.11}), lead to
\begin{equation}\label{eq:6.17}
|R_{n}^{p}(N^{\theta }t)|\leq \widetilde{C}_{p}(1+|t|)
\left(\frac{4\zeta^2}{4-(\alpha_0't+\beta_0')^2}\right)^\frac14
 \big[ |J_{\nu }(N\zeta )|+|J_{\nu +1}(N\zeta)| \big]
\end{equation}
for $t<t_{2}-\sigma$ and for some positive constant $\widetilde{C}_{p}$. Similarly, we have
\begin{equation}\label{eq:6.18}
s_{n+1}^{p}(x)-(A_{n}x+B_{n})s_{n}^{p}(x)+s_{n-1}^{p}(x)=\frac{S_{n}^{p}(x)}{N^{p+\frac{3}{2}}}
\end{equation}
with
\begin{equation}\label{eq:6.19}
|S_{n}^{p}(N^{\theta }t)|\leq \widetilde{C}_{p}(1+|t|)
\left(\frac{4\zeta^2}{4-(\alpha_0't+\beta_0')^2}\right)^\frac14
 \big[ |W_{\nu }(N\zeta )|+|W_{\nu +1}(N\zeta)| \big] .
\end{equation}

We now establish the existence of two solutions $P_{n}(x)$ and $Q_{n}(x)$ of
(\ref{eq:1.7}) satisfying
\begin{equation}\label{eq:6.20}
P_{n}(x)\sim r_{n}^{0}(x)\qquad \text{and}\qquad Q_{n}(x)\sim s_{n}^{0}(x)
\end{equation}
as $n\rightarrow \infty $ for $x=N^\theta t$ with any fixed $t<t_2-\sigma$.
In view of (\ref{eq:5.8}), by using (\ref{eq:5.11}) and (\ref{eq:5.12})
it is easily verified that
\begin{equation}\label{eq:6.21}
r_{n}^{0}(x)\sim \sqrt{\frac{1}{2\pi }}\e^{\nu\pi\i/2}
\left(\frac{4}{(\alpha_0't+\beta_0')^2-4}\right)^\frac14
\e^{-\i N\zeta}
\end{equation}
and
\begin{equation}\label{eq:6.22}
s_{n}^{0}(x)\sim-\sqrt{\frac{2}{\pi}}\e^{-\nu\pi\i/2}
\left(\frac{4}{(\alpha_0't+\beta_0')^2-4}\right)^\frac14
\e^{\i N\zeta}
\end{equation}
as $n\rightarrow \infty$ for any fixed $t<0$.
From (\ref{eq:3.24}), (\ref{eq:3.27}), (\ref{eq:3.28}) and (\ref{eq:4.8}), it follows that
\begin{equation*}
(N+1)\zeta(t_+)-N\zeta(t)\sim
\frac1\i \log\left[\Big(\frac{\alpha't+\beta_0'}{2}\Big)+\i\sqrt{1-\Big(\frac{\alpha't+\beta_0'}{2}\Big)^2}\right]
\end{equation*}
as $n\to\infty$.
If $P_{n}(x)$ and $Q_{n}(x)$ are two solutions of (\ref{eq:1.7}) satisfying (\ref {eq:6.20}),
then we have from (\ref{eq:6.20})--(\ref{eq:6.22})
\begin{equation}\label{eq:6.23}
P_{n+1}(x)Q_{n}(x)-P_{n}(x)Q_{n+1}(x)=P_{n+2}(x)Q_{n+1}(x)-P_{n+1}(x)Q_{n+2}(x)
\end{equation}
and
\begin{equation}\label{eq:6.24}
\begin{aligned}
&P_{n+1}(x)Q_{n}(x)-P_{n}(x)Q_{n+1}(x)\\
=&\lim_{m\rightarrow \infty}[r_{m+1}^{0}(x)s_{m}^{0}(x)-r_{m}^{0}(x)s_{m+1}^{0}(x)]
=\frac{2}{\pi }
\end{aligned}
\end{equation}
for fixed $t<0$. Examining the behavior of  the Bessel functions, one can show that (\ref{eq:6.24}) also holds for fixed $t\geq 0$.
This, in particular, demonstrates that $P_{n}(x)$ and $Q_{n}(x)$ are two linearly independent solutions.

Now define
\begin{equation}\label{eq:6.25}
\varepsilon_n^p(x):=P_n(x)-r_n^p(x)
\qquad\text{and}\qquad
\delta_n^p(x):=Q_n(x)-s_n^p(x).
\end{equation}
We first show that the existence of $Q_n(x)$ to (\ref{eq:1.7})
satisfying (\ref {eq:6.20}) is equivalent to the existence of
$\delta_n^p(x)$ to the summation formula
\begin{equation} \label{eq:6.26}
\begin{aligned}
\delta_n^p(x)=&\sum_{j=n+1}^\infty
\frac{[s_n^p(x)r_j^p(x)-r_n^p(x)s_j^p(x)]S_j^p(x)}{[s_{n+1}^p(x)r_n^p(x)-s_n^p(x)r_{n+1}^p(x)](j+\tau_0)^{p+\frac{3}{2}}}\\
&+\sum_{j=n+1}^\infty\frac{[s_n^p(x)R_j^p(x)-r_n^p(x)S_j^p(x)]\delta_j^p(x)}{[s_{n+1}^p(x)r_n^p(x)-s_n^p(x)r_{n+1}^p(x)](j+\tau_0)^{p+\frac{3}{2}}}.
\end{aligned}
\end{equation}
From (\ref{eq:1.7}) and (\ref{eq:6.18}), we obtain
\begin{equation}\label{eq:6.27}
\delta _{n+1}^{p}(x)-(A_{n}x+B_{n})\delta _{n}^{p}(x)+\delta_{n-1}^{p}(x)
=-\frac{S_{n}^{p}(x)}{N^{p+\frac{3}{2}}}.
\end{equation}
Coupling (\ref{eq:6.10}) and (\ref{eq:6.27}) gives
\begin{equation}
\begin{aligned}
r_n^p(x)\delta_{n+1}^p(x)-r_{n+1}^p(x)\delta_n^p(x)=&
r_{n+1}^p(x) \delta_{n+2}^p(x)-r_{n+2}^p(x)\delta_{n+1}^p(x)\\
&+\frac{S_{n+1}^p(x)r_{n+1}^p(x)+R_{n+1}^p(x)\delta_{n+1}^p(x)}{(N+1)^{p+3/2}},
\end{aligned}  \label{eq:6.28}
\end{equation}
which leads to
\begin{equation}\label{eq:6.29}
\begin{aligned}
r_n^p(x)\delta_{n+1}^p(x)-r_{n+1}^p(x)\delta_n^p(x)
=&r_{m+1}^p(x)\delta_{m+2}^p(x)-r_{m+2}^p(x)\delta_{m+1}^p(x)\\
&+\sum_{j=n+1}^{m+1}\frac{S_j^p(x)r_j^p(x)+R_j^p(x)\delta_j^p(x)}{(j+\tau_0)^{p+3/2}}.
\end{aligned}
\end{equation}
If $Q_n(x)$ satisfies (\ref{eq:6.20}), then $\delta_n^0(x)=o\left(s_n^0(x) \right)$ as $n \to \infty$.
Since $\delta_n^p(x)=\delta_n^0(x)+s_n^0(x)-s_n^p(x)$ by (\ref{eq:6.25}),
as well as $\widetilde{A}_s(\zeta)$ and $\widetilde{B}_s(\zeta)$ are bounded for $s\geq 0$,
on account of (\ref{eq:6.8}) we also have $\delta_n^p(x)=o(s_n^0(x))$ as $n \to \infty$.
In view of (\ref{eq:6.21}) and (\ref{eq:6.22}), we have
\begin{equation}\label{eq:6.30}
r_{m+1}^p(x)\delta_{m+2}^p(x)-r_{m+2}^p(x)\delta_{m+1}^p(x) \to 0
\end{equation}
as $m \to\infty$. Hence, by letting $m\rightarrow \infty $ in (\ref{eq:6.29}), we obtain
\begin{equation} \label{eq:6.31}
r_{n}^{p}(x)\delta_{n+1}^{p}(x)-r_{n+1}^{p}(x)\delta_{n}^{p}(x)
=\sum_{j=n+1}^{\infty }\frac{S_{j}^{p}(x)r_{j}^{p}(x)+R_{j}^{p}(x)\delta_{j}^{p}(x)}{(j+\tau_0)^{p+3/2}}.
\end{equation}
In a similar manner, we also have
\begin{equation} \label{eq:6.32}
s_{n}^{p}(x)\delta_{n+1}^{p}(x)-s_{n+1}^{p}(x)\delta_{n}^{p}(x)
=\sum_{j=n+1}^{\infty }\frac{S_{j}^{p}(x)s_{j}^{p}(x)+S_{j}^{p}(x)\delta_{j}^{p}(x)}{(j+\tau_0)^{p+3/2}}
\end{equation}
and
\begin{equation}\label{eq:6.33}
r_{n+1}^p(x)s_n^p(x)-s_{n+1}^p(x)r_n^p(x)=\frac2{\pi}
+\sum_{j=n+1}^{\infty }\frac{S_{j}^{p}(x)r_{j}^{p}(x)-R_{j}^{p}(x)s_{j}^{p}(x)}{(j+\tau_0)^{p+3/2}}.
\end{equation}
Then (\ref{eq:6.26}) follows from (\ref{eq:6.31}) and (\ref{eq:6.32}).
The existence of a solution $\{\delta _n^p(x)\}_{n=1}^\infty$ to
(\ref {eq:6.26}) is proved by using the method of successive approximation.
Starting with $\delta_{n,0}^p(x)=0$, we define $\delta_{n,k}^p(x)$ by
\begin{equation}\label{eq:6.34}
\begin{aligned}
\delta_{n,k}^p(x)=
&\sum_{j=n+1}^\infty\frac{[s_n^p(x)r_j^p(x)-r_n^p(x)s_j^p(x)]S_j^p(x)}{[s_{n+1}^p(x)r_n^p(x)-s_n^p(x)r_{n+1}^p(x)](j+\tau_0)^{p+\frac{3}{2}}}\\
&+\sum_{j=n+1}^\infty\frac{[s_n^p(x)R_j^p(x)-r_n^p(x)S_j^p(x)]\delta_{j,k-1}^p(x)}{[s_{n+1}^p(x)r_n^p(x)-s_n^p(x)r_{n+1}^p(x)](j+\tau_0)^{p+\frac32}}
\end{aligned}
\end{equation}
for $k\geq1$ recursively. We shall show that for fixed $p$ and sufficiently large but also fixed $n$,
the sequence $\{\delta_{n,k}^p(x)\}_{k\geq0}$ is convergent as $k\to\infty$.
Since $\widetilde{A}_s(\zeta)$ and $\widetilde{B}_s(\zeta)$ are bounded for $t<t_2-\sigma$,
it follows from (\ref{eq:6.6})--(\ref{eq:6.9}) that
\begin{equation}\label{eq:6.35}
|r_n^p(x)|\leq
CN^\frac{1}{2}
\left(\frac{4\zeta^2}{4-(\alpha_0't+\beta_0')^2}\right)^\frac14
\big[|J_\nu(N\zeta)|+|J_{\nu+1}(N\zeta)|\big]
\end{equation}
and
\begin{equation}\label{eq:6.36}
|s_n^p(x)|\leq
CN^\frac{1}{2}
\left(\frac{4\zeta^2}{4-(\alpha_0't+\beta_0')^2}\right)^\frac14
\big[|W_\nu(N\zeta)|+|W_{\nu+1}(N\zeta)|\big]
\end{equation}
for some positive constant $C$. Furthermore, by virtue of the
behaviors of $ J_\nu$ and $W_\nu$, we have from (\ref{eq:6.17}), (\ref{eq:6.19}), (\ref{eq:6.35}) and
(\ref{eq:6.36}) that
\begin{equation}\label{eq:6.37}
|R_n^p(x)s_n^p(x)|\le M'N^\frac{1}{2}
\quad\text{and}\quad
|S_n^p(x)r_n^p(x)|\le M'N^\frac{1}{2}
\end{equation}
for some positive constant $M'$.
It then follows from (\ref{eq:6.33}) that
\begin{equation*}
\left|s_{n+1}^p(x)r_n^p(x)-r_{n+1}^p(x)s_n^p(x)+\frac2\pi\right|
\leq\frac{2M'}{p}\cdot\frac{1}{N^p}.
\end{equation*}
Since the right-hand side tends to zero, this estimate gives
\begin{equation}\label{eq:6.38}
|s_{n+1}^p(x)r_n^p(x)-r_{n+1}^p(x)s_n^p(x)|>\frac{1}{\pi}
\end{equation}
for large $n$. A combination of (\ref{eq:6.34}) and (\ref{eq:6.35})--(\ref{eq:6.38}) yields
\begin{equation*}
\begin{aligned}
|\delta_{n,1}^p(x)|\le&\pi\sum_{j=n+1}^\infty
\frac{[|s_n^p(x)r_j^p(x)|+|r_n^p(x)s_j^p(x)|]|S_j^p(x)|}{(j+\tau_0)^{p+\frac{3}{2}}}
\\ \leq& M''
\left(\frac{4\zeta^2}{4-(\alpha_0't+\beta_0')^2}\right)^\frac14
\sum_{j=n+1}^\infty\frac{N^\frac{1}{2}}{(j+\tau_0)^{p+1}}\big[|W_\nu(N\zeta)|+|W_{\nu+1}(N\zeta)|\big],
\end{aligned}
\end{equation*}
where we have also made use of the monotonicity properties and the asymptotic behaviors of the modified Bessel functions
$I_\nu$ and $K_\nu$; see (\ref{eq:5.8}), (\ref{eq:5.11}) and (\ref{eq:5.12}).
Hence,
\begin{equation} \label{eq:6.39}
|\delta_{n,1}^p(x)| \leq
\frac{M''}{p}\frac{1}{N^{p-\frac{1}{2}}}
\left(\frac{4\zeta^2}{4-(\alpha_0't+\beta_0')^2}\right)^\frac14
\big[|W_\nu(N\zeta)|+|W_{\nu+1}(N\zeta)|\big].
\end{equation}
Similarly, we can prove by induction that
\begin{equation}\label{eq:6.40}
\begin{aligned}
&|\delta_{n,k}^p(x)-\delta_{n,k-1}^p(x)|\\
\leq&
\left(\frac{M''}{p} \frac{1}{N^{p+1}}\right)^kN^\frac32
\left(\frac{4\zeta^2}{4-(\alpha_0't+\beta_0')^2}\right)^\frac14
\big[|W_\nu(N\zeta)|+|W_{\nu+1}(N\zeta)|\big],
\end{aligned}
\end{equation}
from which it follows that
\begin{equation}\label{eq:6.41}
\delta_{n,k}^p(x)=\sum_{m=1}^k[\delta_{n,m}^p(x)-\delta_{n,m-1}^p(x)]
\end{equation}
converges as $k\to\infty$, for all $n>2M''-\tau_0$. Clearly,
the limiting function $\delta_n^p(x)$ satisfies (\ref{eq:6.26}).
Thus $Q_n(x)=s_n^p(x)+\delta_n^p(x)$ is a solution of (\ref{eq:1.7})
satisfying (\ref{eq:6.20}). Furthermore, we have from
(\ref{eq:6.39})--(\ref{eq:6.41}) that
\begin{equation*}
|\delta_n^p(x)|\le \frac{M'_p}{N^{p-\frac{1}{2}}}
\left(\frac{4\zeta^2}{4-(\alpha_0't+\beta_0')^2}\right)^\frac14
\big[|W_\nu(N\zeta)|+|W_{\nu+1}(N\zeta)|\big].
\end{equation*}
Recall that $\widetilde{A}_s(\zeta)$ and $\widetilde{B}_s(\zeta)$ are bounded for $t<t_2-\sigma$.
By taking an extra term in the expansion in (\ref{eq:6.2}), we have
\begin{equation}\label{eq:6.42}
\begin{aligned}
|\delta_n^p(x)|
\leq&|\delta_n^{p+1}(x)|+\frac{M''_{p}}{N^{p+\frac12}}
\left(\frac{4\zeta^2}{4-(\alpha_0't+\beta_0')^2}\right)^\frac14\\
&\hspace{2.2cm}\times
\big[|W_\nu(N\zeta)||\widetilde{A}_{p+1}(\zeta)|+|W_{\nu+1}(N\zeta)||\widetilde{B}_{p+1}(\zeta)|\big]\\
\leq&
\frac{M_p}{N^{p+\frac{1}{2}}}
\left(\frac{4\zeta^2}{4-(\alpha_0't+\beta_0')^2}\right)^\frac14
\big[|W_\nu(N\zeta)|+|W_{\nu+1}(N\zeta)|\big]
\end{aligned}
\end{equation}
for some positive constants $M''_p$ and $M_p$, and (\ref{eq:2.21}) follows.
(Note that there is an extra factor of $N^\frac12$ outside the square brackets in (\ref{eq:2.19}),
so the error terms $\delta_n^p$ in (\ref{eq:2.21}) and (\ref{eq:6.42}) have a slightly different meaning.)

In a similar manner as we have done for $\delta_n^p(x)$, it can be proved that
\begin{equation}\label{eq:6.43}
|\varepsilon_n^p(x)| \leq\frac{M_p}{N^{p+\frac12}}
\left(\frac{4\zeta^2}{4-(\alpha_0't+\beta_0')^2}\right)^\frac14
\big[|J_\nu(N\zeta)|+|J_{\nu+1}(N\zeta)|\big],
\end{equation}
and hence (\ref{eq:2.20}) follows.
This completes the proof of Theorem~\ref{thm:1}.

\section{An example}

As an illustration,
we consider the orthogonal polynomials $p_n(x)$ associated with the  Laguerre-type weights
$w(x)=x^\alpha\exp(-q_mx^m)$, $x>0$, $\alpha>-1$, $q_m>0$.
This weight is a Freud weight on the half line;
the asymptotics of $p_n(x)$ have been investigated by Vanlessen~\cite{Vanlessen} via a Riemann-Hilbert method.
For the Freud weight on the whole real line, the asymptotics of the associated polynomials can be found in \cite{Deift,Krie,WongZhang}.
The notation $p_n(x)=\gamma_nx^n+\cdots$ with leading coefficient $\gamma_n>0$
denotes the $n$-th degree orthonormal polynomial with respect to $w(x)$.
It is known that the sequence of $p_n(x)$ satisfies the three-term recurrence relation
\begin{equation}\label{eq:7.1}
xp_n(x)=b_np_{n+1}(x)+a_np_n(x)+b_{n-1}p_{n-1}(x),\qquad n=1,2,\dots,
\end{equation}
with $p_{-1}(x)=0$ and $p_0(x)=\gamma_0>0$. The recurrence coefficients $a_n$ and $b_{n-1}$
have the asymptotic expansions
\begin{equation}\label{eq:7.2}
b_{n-1}\sim n^{\frac1m}r_m\left\{\frac14+\frac{\alpha}{8m}\frac 1n+\frac{c_2}{n^2}+\O{\frac1{n^3}}\right\}
\end{equation}
and
\begin{equation}\label{eq:7.3}
a_n\sim n^{\frac1m}r_m\left\{\frac12+\frac{\alpha+1}{4m}\frac 1n+\frac{d_2}{n^2}+\O{\frac1{n^3}}\right\},
\end{equation}
where $c_2$ and $d_2$ are some constants and
\begin{equation}\label{eq:7.4}
r_m=\left(\frac12mq_m\prod_{j=1}^m\frac{2j-1}{2j}\right)^{-1/m};
\end{equation}
see \cite{Vanlessen}.
We define  a sequence $\{K_n\}$ by
$$
K_n=n^{-\frac1{2m}}\prod_{l=0}^\infty\Big[\Big(\frac{n+2l}{n+2l+2}\Big)^\frac1{2m}\cdot \frac{b_{n+2l+1}}{b_{n+2l}}\Big].
$$
Note that $K_{n+1}/K_{n-1}=b_{n-1}/b_n$ and $K_n\sim n^{-\frac 1{2m}}$. We then put
\begin{equation}\label{eq:7.5}
A_n\equiv -\frac1{b_n}\frac{K_n}{K_{n+1}},\quad
B_n\equiv \frac{a_n}{b_n}\frac{K_n}{K_{n+1}},
\quad
\mathcal{P}_n(x)\equiv (-1)^n[w(x)]^{\frac12}p_n(x)\frac{1}{K_n}.
\end{equation}
The three-term recurrence relation in (\ref{eq:7.1}) now becomes
\begin{equation}\label{eq:7.6}
\mathcal{P}_{n+1}(x)+\mathcal{P}_{n-1}(x)=(A_nx+B_n)\mathcal{P}_n(x).
\end{equation}
A slight computation gives the following asymptotic expansions
\begin{equation}\label{eq:7.7}
\frac{K_n}{K_{n+1}}\sim 1+\frac{1}{2m}\frac1n-\frac{4m+2\alpha m-1}{8m^2}\frac{1}{n^2}+\O{\frac1{n^3}},
\end{equation}
\begin{equation}\label{eq:7.8}
A_n\sim n^{-\frac1m}r_m^{-1}\left\{-4+\frac{2+2\alpha}{m}\frac{1}{n}+\O{\frac1{n^2}}\right\}
\end{equation}
and
\begin{equation}\label{eq:7.9}
B_n\sim 2+\frac{(4m^2-4m+1)\alpha^2-m^2}{4m^2}\frac{1}{n^2}+\O{\frac1{n^3}}.
\end{equation}
In terms of our notations, $\theta=\frac1m$ and $\tau_0=\frac{\alpha+1}{2}$. Set $N=n+\tau_0$. We then have
\begin{equation}\label{eq:7.10}
A_n\sim N^{-\frac1m}\sum_{s=0}^\infty\frac{\alpha_s'}{N^s},\qquad
B_n\sim \sum_{s=0}^\infty\frac{\beta_s'}{N^s},
\end{equation}
where
$$
\alpha_0'=-4r_m^{-1},\quad\alpha_1'=0,\quad\beta_0'=2,\quad\beta_1'=0,\quad \beta_2'=\frac{(2m-1)^2\alpha^2-m^2}{4m^2}.
$$
Let $x=N^\theta t$. To simplify matters, we further set $t=r_m z$ and $\omega_n=N^\theta r_m $, so that $x=\omega_n z$.
Since $\theta=\frac1m\neq0,2$ and $\beta_0'=2$, the two transition points are $t_1=0$ and $t_2=r_m$; cf. (\ref{eq:1.11}).
We can now apply Theorem~\ref{thm:1}. In terms of the new scaled variable $z$, the two transition points are $z_1=0$ and $z_2=1$.
From (\ref{eq:2.15}), it can be shown that $\nu=\alpha$, and by (\ref{eq:4.9})
\begin{equation}\label{eq:7.11}
\zeta(z)=\arccos(1-2z)+\frac{2\sqrt{z(1-z)}}{2m-1}{}_2F_1\left(1,1-m;\frac32-m;z\right),\quad z<1-\sigma.
\end{equation}

\begin{rem}
For any given integer $m>0$, the hypergeometric function in (\ref{eq:7.11}) is a polynomial of  degree $(m-1)$.
For example, when $m=1$, it is $1$; when $m=2$, it is $2z+1$; and when $m=3$, it is $\frac83z^2+\frac43z+1$.
\end{rem}

From Theorem~\ref{thm:1}, it follows that we have two linearly independent solutions $P_n(x)$ and $Q_n(x)$
given in (\ref{eq:2.18}) and (\ref{eq:2.19}).
Hence,
\begin{equation}\label{eq:7.12}
\mathcal{P}_n(\omega_nz)= C_1(x) P_n(x)+C_2(x)Q_n(x),
\end{equation}
where $C_1(x)$ and $C_2(x)$ are functions depending only on $x$.
To determine $C_1(x)$ and $C_2(x)$,
we note that in terms of our notations, equation (2.12) in \cite{Vanlessen} reads
\begin{equation}\label{eq:7.13}
[w(\beta_n z)]^{\frac12}p_n(\beta_n z)\sim \sqrt{\frac{2}{\pi\beta_n}}
\frac{[(2z-1)+2\sqrt{z(z-1)}]^{\frac{\alpha+1}{2}}}{2z^\frac14(z-1)^\frac14}
\exp\{-\i n\zeta(z)\},
\end{equation}
where $\beta_n=n^{\frac 1m}r_m$.
This asymptotic formula holds uniformly for $z$ in any complex subset $K\subset\mathbb C\setminus[0,1]$.
Note that our uniform asymptotic expansion of $\mathcal{P}_n(\omega_nz)$ (\ref{eq:7.12}) is also valid in a neighbourhood of any subinterval of $(-\infty,1-\sigma]$ in the complex $z$-plane.
Since $\omega_n=(1+\tau_0/n)^\frac1m \beta_n$, matching these two asymptotic formulas in an overlapping region,
and using the behaviors of $J_\nu(x)$ and $Y_\nu(x)$ for large $x$, we obtain $C_1(x)=1$ and $C_2(x)=0$. That is
\begin{equation}\label{eq:7.14}
\begin{aligned}
\mathcal{P}_n(\omega_n z)=&
N^\frac12\left(\frac{\zeta^2(z)}{z(1-z)}\right)^\frac14
\left[J_\alpha(N\zeta)\sum_{s=0}^p\frac{
\widetilde{A}_s(\zeta)}{N^s}\right.\\
&\hspace{3.4cm}+\left.J_{\alpha+1}(N\zeta)
\sum_{s=0}^p\frac{\widetilde{B}_s(\zeta)}{N^s}
+\varepsilon_n^p\right],
\end{aligned}
\end{equation}
where $\widetilde{A}_0(\zeta)=1$, $\widetilde{B}_0(\zeta)=0$, and
\begin{equation}\label{eq:7.15}
|\varepsilon_n^p|\leq
\frac{M_p}{N^{p+1}}
\big[|J_\alpha(N\zeta)|+|J_{\alpha+1}(N\zeta)|\big]
\end{equation}
for all $z\leq 1-\sigma$.
We can also use the result in \cite{ww2} to derive a uniform asymptotic expansion for $\mathcal{P}_n(x)$ near the turning point $z_2=1$,
and we have
\begin{equation}\label{eq:7.16}
\begin{aligned}
(-1)^n\mathcal{P}_n(\omega_n z)=
N^\frac16\left(\frac{\eta(z)}{z(z-1)}\right)^\frac14
&\left[\Ai(N^\frac23\eta)\sum_{s=0}^p\frac{
\overline{A}_s(\eta)}{N^{s}}\right.\\
&\hspace{0.4cm}+\left.\Ai'(N^\frac23\eta)
\sum_{s=0}^p\frac{\overline{B}_s(\eta)}{N^{s+1/3}}
+\overline{\varepsilon_n^p}\right]
\end{aligned}
\end{equation}
for $z\geq\sigma>0$, where $\overline{A}_0(\eta)=1$, $\overline{B}_0(\eta)=0$,
\begin{equation}\label{eq:7.17}
\frac23[-\eta(z)]^\frac{3}{2}
=\arccos(2z-1)-\frac{2\sqrt{z(1-z)}}{2m-1}\cdot{}_2F_1\left(1,1-m;\frac32-m;z\right)
\end{equation}
for $\sigma\leq z\leq 1$,
and $\eta(z)$ analytically continued to $\mathbb C\setminus\{(-\infty,0]\cup[1,-\i\infty)\}$
such that $\eta(z)>0$ for $z>1$. The error estimation is given by
\begin{equation}\label{eq:7.18}
\big|\overline{\varepsilon_n^p}\big|\leq
\frac{M_p}{N^{p+1}}
\widetilde{\Ai}(N^\frac23\eta)
\end{equation}
with the modulus function $\widetilde{\Ai}(N^\frac23\eta)$ as defined in \cite[eq. (7.10)]{ww2}.
Note that the two results in (\ref{eq:7.14}) and (\ref{eq:7.16}) together cover the whole real line.

\begin{rem} In the special case of $m=1$ and $q_m=1$, we get two uniform asymptotic expansions for Laguerre polynomials,
which agree with the results obtained by steepest decent method for integrals \cite{fw}
or WKB approximations for differential equations \cite[Chaps. 11 and 12]{Olver}.

For orthogonal polynomials associated with the weight $x^\alpha \exp(-Q(x))$, $x>0$,
$\alpha>-1$ and $Q(x)$ is a polynomial of $m$-th degree with positive leading coefficient,
the asymptotic expansions of the recurrence coefficients will be in powers of $1/n^{\frac1m}$,
instead of in powers series of $1/n$; cf. (\ref{eq:1.8}).
In such a case, one can modify the method provided in this paper
to get a pair of linearly independent solutions to the three-term recurrence relation given in Theorem~\ref{thm:1},
which are also in terms of Bessel functions near the transition point $t_1=0$.
However, the asymptotic expansions for these solutions are on longer in powers of $1/N$,
but in powers of $1/N^{\frac1m}$.
This result will lead to Bessel-type asymptotic expansions given by Vanlessen in \cite{Vanlessen}.
\end{rem}

As another example, confluent hypergeometric functions $M(a,b;x)$ and $\Gamma(a-b+1)\cdot U(a,b;x)$ satisfy the three-term recurrence relation in $a$:
\begin{equation}\label{eq:7.19}
ay(a+1,b;x)+(a-b)y(a-1,b;x)=(2a-b+x)y(a,b;x).
\end{equation}
A straightforward calculation shows that equation (\ref{eq:7.19}) corresponds to our case for $\theta=1$, $\alpha_0=1$ and $\beta_0=2$.
A direct application of Theorem~\ref{thm:1} yields Bessel-type expansions as $a\to+\infty$.
Replacing $a$ by $-a$ in (\ref{eq:7.19}), one can obtain again, by the main theorem, Bessel-type expansions for $M(a,b;x)$ or $U(a,b;x)$
as $a\to-\infty$; see \cite[\S 13.8($\mathrm{iii}$)]{NIST}.

\section*{Acknowledgements}
The work of Y. Li is supported in part by the HKBU Strategic Development Fund
and a start-up grant from Hong Kong Baptist University.

\end{document}